\documentclass{article}
\usepackage[utf8]{inputenc}
\usepackage{amsmath}
\usepackage{amssymb}
\usepackage{amsthm}
\usepackage{bm}
\usepackage{verbatim}
\usepackage{graphicx}
\usepackage{subcaption}
\usepackage[shortlabels]{enumitem}
\usepackage{xcolor}
\usepackage{float}

\newtheorem{theorem}{Theorem}[section]

\newtheorem{lemma}[theorem]{Lemma}
\newtheorem{remark}[theorem]{Remark}

\newtheorem{assumption}[theorem]{Assumption}

\newcommand{\KK}{{K_h}} 
\newcommand{\AAA}{{\mathsf{A}}}
\newcommand{\AAK}{\mathsf{A}_k}

\newcommand{\WW}{{\mathsf{W}}}
\newcommand{\WWK}{\mathsf{W}_k}
\newcommand{\PP}{\mathsf{P}}
\newcommand{\PPK}{\mathsf{P}_k}
\newcommand{\VV}{{\mathsf{V}}}
\newcommand{\SSa}{{\mathsf{S}_{\beta}}}
\newcommand{\UU}{{\mathsf{U}}}
\newcommand{\SI}{{\mathsf{\Sigma}}}
\newcommand{\SIK}{\mathsf{\Sigma}_k}

\newcommand{\xx}{{\mathbf{x}}}
\newcommand{\yy}{{\mathbf{y}}}

\newcommand{\bb}{{\mathbf{b}}}
\newcommand{\ee}{{\mathbf{e}}}
\newcommand{\nn}{{\mathbf{n}}}
\newcommand{\halpha}{{\beta}}
\newcommand{\hxx}{{\hat{\mathbf{x}}}}
\newcommand{\Rn}{{\mathbb{R}^n}}

\newcommand{\Rnn}{{\mathbb{R}^{n\times n}}}
\newcommand{\Rmn}{{\mathbb{R}^{m\times n}}}
\newcommand{\gam}{{\gamma_{j,\alpha}}}

\newcommand{\gamn}{{\gamma_{j,\alpha,\eta}}}

\newcommand{\nullspace}[1]{{\mathcal{N}(#1)}}

\newcommand{\pphi}{{\mathbf{\phi}}}

\DeclareMathOperator*{\argmin}{arg\,min}
\DeclareMathOperator*{\argmax}{arg\,max}

\title{Weighted sparsity regularization for source identification for elliptic PDEs}
\author{Ole L{\o}seth Elvetun\thanks{Faculty of Science and Technology, Norwegian University of Life Sciences, P.O. Box 5003, NO-1432 {\AA}s, Norway. Email: ole.elvetun@nmbu.no.} and Bj{\o}rn Fredrik Nielsen\thanks{Faculty of Science and Technology, Norwegian University of Life Sciences, P.O. Box 5003, NO-1432 {\AA}s, Norway. Email: bjorn.f.nielsen@nmbu.no. Nielsen's work was supported by The Research Council of Norway, project number 239070.}}

\begin{document}

\maketitle

\begin{abstract}
    This investigation is motivated by PDE-constrained optimization problems arising in connection with \textcolor{black}{electrocardiograms (ECGs) and electroencephalography (EEG)}.  Standard sparsity regularization does not necessarily produce adequate results for these applications because only boundary data/observations are available for the identification of the unknown source, which may be interior. We therefore study a weighted $\ell^1$-regularization technique for solving inverse problems when the forward operator has a significant null space. In particular, we prove that a sparse source, regardless of whether it is interior or located at the boundary, can be exactly recovered with this weighting procedure as the regularization parameter $\alpha$ tends to zero. 
    Our analysis is supported by numerical experiments for cases with one and several local sources. The theory is developed in terms of Euclidean spaces, and our results can therefore be applied to many problems.  
\end{abstract}

\noindent {\bf 2010 Mathematics Subject Classification:} 35R30, 47A52, 65F22. \\

\noindent {\bf Keywords:}
Sparsity regularization, inverse source problems, PDE-constrained optimization, null space.

\section{Introduction}
Consider the challenge of identifying a sparse source in an elliptic PDE from Dirichlet boundary data: 
    \begin{equation}\label{eq1}
        \min_{(f,u) \in F_h \times H^1(\Omega)} \left\{ \frac{1}{2}\|u - d\|_{L^2(\partial\Omega)}^2 + \alpha \sum_i w_i |(f, \phi_i)_{L^2(\Omega)}| \right\}
    \end{equation}
    subject to 
    \begin{equation}\label{eq2}
    \begin{split}
        -\Delta u + \epsilon u &= f \quad \mbox{in } \Omega, \\
        \frac{\partial u}{\partial \nn} &= 0  \quad \mbox{on } \partial \Omega, 
    \end{split}
    \end{equation} 
where we employ weighted $\ell^1$-regularization in \eqref{eq1}. Here, $\{ \phi_1, \phi_2, \ldots, \phi_n \}$ is an $L^2$-orthonormal basis for $F_h$, $\{ w_i \}$ are positive weights, $\alpha > 0$ is a regularization parameter, \textcolor{black}{$d$ represents the Dirichlet boundary data}, $\epsilon$ is a parameter, $\nn$ denotes the outwards pointing unit normal vector of the boundary $\partial \Omega$ of the bounded domain $\Omega$, and $f$ is the unknown source. 

\textcolor{black}{Note that we, for the sake of simplicity, consider a finite dimensional control/source space $F_h$. Since the goal is to recover spatially sparse solutions, the basis functions $\{ \phi_1, \phi_2, \ldots, \phi_n \}$ should have small and local support. In the infinite dimensional setting one would thus typically search for point-sources or Dirac measures. As shown in \cite{casas2012approximation}, this leads to a number of subtle mathematical issues, even in the unweighted case.}

Variants of \textcolor{black}{\eqref{eq1}-\eqref{eq2}} appear in many applications, such as in crack determination, in the inverse ECG problem and in \textcolor{black}{EEG investigations}. Consequently, it has received much attention from researchers, see, e.g.,  \cite{babda09,cheng15,Han11,Het96,hinze19,BIsa05,kun94,ring95,song12}. However, in these studies the authors did not use $\ell^1$-regularization. A more detailed description of previous investigations is presented in \cite{Elv20}. 

Regularization of inverse problems with sparsity promoting methods has increased in popularity in 
recent years \cite{daubechies04, grasmair08, Jin_2012, Jin_2017,lorenz08, Lu_2019, Flemming16}. In this context, the notion of sparsity, with respect to a given basis $\{\phi_i\}$ for the control space, means that the inverse solution $f^*$ can be represented using very few of the basis functions. That is, if we expand the inverse solution $f^*$ as $f^*(x) = \sum_i f_i^*\phi_i(x)$, then $f_i^* = (f^*,\phi_i)_{L^2(\Omega)} \neq 0$ only for very few of the basis functions. If we have $s$ such non-zero components, we say that the solution $f^*$ is $s$-sparse.

In \textit{compressed sensing} the study of $\ell^1$-regularization for the exact recovery of a sparse source from noise free data has been studied in detail \cite{candes06,candes05,donoho03}. To elaborate the main findings in \cite{candes05}, let us assume that $\AAA$ is a matrix with a significant null space. If $\bb^\dagger$ is generated from a sparse source $\xx^\dagger$, i.e., $\bb^\dagger = \AAA\xx^\dagger$, then the minimizer of
\begin{equation*}
    \min_{\xx\in\Rn} \|\xx\|_1 \quad \textnormal{subject to} \quad \AAA\xx = \bb^\dagger,
\end{equation*}
is the true sparse solution $\xx^\dagger$ when a certain assumption \textcolor{black}{on $\AAA$}, known as the \textit{restricted isometry property}, is fulfilled.

In \cite{grasmair10} the authors unified this result with the theory developed by the inverse problem community. In particular, they showed that the commonly used range condition, combined with an additional restricted injectivity condition on the forward operator, are weaker assumptions than the previously mentioned \textit{restriced isometry property}, and they proved that the former conditions are the weakest which admit linear convergence rates for the regularized problem
\begin{equation*}
    \min_{\xx\in\Rn} \left\{ \frac{1}{2} \|\AAA\xx - \bb\|_2^2 + \alpha \|\xx\|_1 \right\}.
\end{equation*}

\textcolor{black}{Numerical experiments indicate that the identification criteria, mentioned in the previous paragraph, are not fulfilled by the forward/transfer matrix $\AAA$ associated with \eqref{eq1}-\eqref{eq2}, 
see Figure \ref{fig:ex0}. \textcolor{black}{(Further details about the matrix $\AAA$ are presented in the next section and in Appendix \ref{app:discretization}.)} This is the main motive for the present study. Observe also that the inverse solution in panel (b) in Figure \ref{fig:ex0} only is nonzero close to where observations are made, i.e., close to the boundary $\partial \Omega$ of $\Omega$. \textcolor{black}{This type of "behaviour" was not only observed with $\alpha=10^{-4}$, but in every experiment we performed with standard sparsity regularization. In view of the analysis presented in Appendix \ref{app:standard}, see also Proposition 2.3 in \cite{casas2012approximation}, this is not surprising.}}

\begin{figure}[h]
    \centering
    \begin{subfigure}[b]{0.45\linewidth}        
        \centering
        \includegraphics[width=\linewidth]{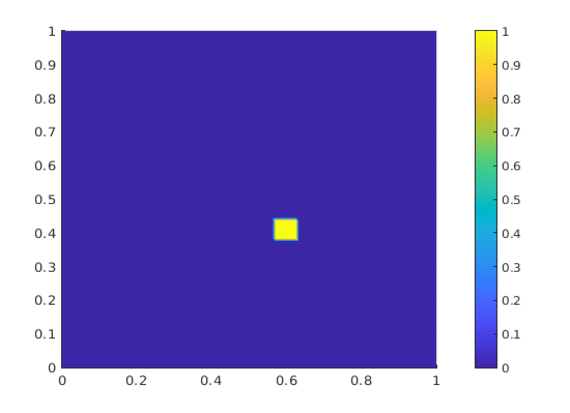}
        \caption{True source.}
    \end{subfigure}
    \begin{subfigure}[b]{0.45\linewidth}        
        \centering
        \includegraphics[width=\linewidth]{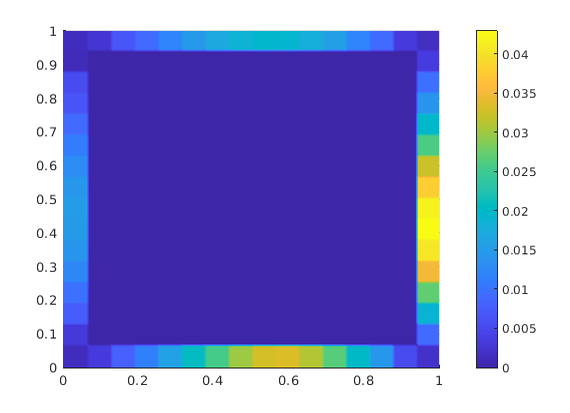}
        \caption{Inverse solution.}
    \end{subfigure}\par
    \caption{Panel (b) shows the outcome of attempting to use \eqref{eq1}-\eqref{eq2}, with $w_i=1$ for $i=1,2,\ldots,n$ and $\alpha = 10^{-4}$, to recover the true interior source depicted in panel (a). That is, the Dirichlet boundary data $d$ in \eqref{eq1} was generated by the true source (a).} 
    \label{fig:ex0}
\end{figure}


Many researchers have explored the use of sparsity regularization in connection with PDEs, see, e.g., \cite{casas2012approximation,casas2013parabolic,Golmohammadi2018,herman2020randomization,Xiang_2020,li2019sparse,Rudy2009,Wang2012}. However, as far as the authors know, the use of weighted $\ell^1$-regularization to solve inverse source problems for PDEs, employing only boundary measurements, has not previously been attempted. In \cite{mehr_2015} the authors apply an iterative reweighted $\ell^1$-regularization technique \cite{candes2008enhancing} to recover sources, but their approach differs significantly from ours and involve data recorded at equidistant locations inside the domain. 

We will start with a brief motivation for developing weighted sparsity promoting regularization for ECG and EEG applications in Section \ref{sec:motivation}. In Section \ref{sec:analysis} we prove that our methodology can recover, without any errors or blurring, a single local source, regardless of whether it is interior or at the boundary. We also analyze a problem for which it is impossible to recover two separate sources. This implies that we can not guarantee, in general, that our approach can recover multiple local sources. However, in Section \ref{sec:numerical_experiments} many numerical experiments are presented, for both single and multiple sources, and we observe that several sources often can be successfully recovered.

We would like to emphasize that, even though our motivation originates from inverse source problems for elliptic PDEs, the general theory presented in this paper is applicable to any linear finite-dimensional inverse problem where the forward operator has a non-trivial null space.

\textcolor{black}{One may regard this paper to be follow-up work to \cite{Elv20}: In \cite{Elv20} weighted Tikhonov regularization is proposed and analyzed, and in this text we employ the same weight-matrix in connection with sparsity regularization. However, since both the results and the analysis of the sparsity approach differs significantly from the investigation of the quadratic regularization, a separate study is needed. More precisely, weighted Tikhonov regularization yields a "blurred/smooth" reconstruction of internal sources and the cost-functional is differentiable, whereas weighted sparsity regularization enables perfect recovery of a single source but the objective function is not differentiable.} 

\section{\textcolor{black}{Preliminaries}}\label{sec:prelim}
\textcolor{black}{
To make the forthcoming results applicable to more general finite-dimensional problems, we will present our analysis in terms of Euclidean spaces. First, however, we will study \eqref{eq1}-\eqref{eq2} in greater detail and derive the associated fully discretized problem.}

\textcolor{black}{
We interpret \eqref{eq2} in the following weak sense: Let $\Omega \subset \mathbb{R}^\nu$, $\nu=1,2,3$, be a Lipschitz domain and assume that $f \in F_h \subset L^2(\Omega)$ is given. A function $u \in H^1(\Omega)$ is a solution of \eqref{eq2} if
\begin{equation}
    \int_\Omega \nabla u \cdot \nabla v \ dx +  \epsilon\int_\Omega uv \ dx = \int_\Omega fv \ dx, \quad \forall v \in H^1(\Omega). \label{eq:pre1}
\end{equation}
From standard elliptic PDE-theory \cite{evans98} it follows that there exists a unique solution to \eqref{eq:pre1} which depends continuously on $f \in F_h \subset L^2(\Omega)$. Since the trace operator $T: H^1(\Omega) \rightarrow L^2(\partial\Omega), \, u \mapsto u|_{\partial\Omega}$, is continuous, we can conclude that the forward operator
\begin{equation}
 \KK: F_h \rightarrow L^2(\partial\Omega), \quad f \mapsto u(f)|_{\partial\Omega}, \label{eq:pre2}
\end{equation}
associated with \eqref{eq1}-\eqref{eq2}, is well-defined and continuous.}

\textcolor{black}{
We can now formulate the problem
\begin{equation} \label{eq:pre3}
    \min_{f \in F_h} \underbrace{\left\{ \frac{1}{2}\|\KK f - d\|^2_{L^2(\partial\Omega)} + \alpha \sum_i w_i |(f, \phi_i)_{L^2(\Omega)}| \right\}}_{= \mathcal{G}(f)},
\end{equation}
which is equivalent to \eqref{eq1}-\eqref{eq2}. Continuity of the cost functional $\mathcal{G}$ follows immediately from the continuity of the norm and, since we assume that $w_i > 0$, $i=1, 2,\ldots, n$, $\mathcal{G}$ is also coercive, provided that $\alpha > 0$. Standard optimization theory thus yields  that \eqref{eq:pre3} has a global minimizer, see, e.g., \cite{Peressini88}.}

\textcolor{black}{
Since $\KK$ is a linear mapping from a finite-dimensional space onto a finite dimensional subspace of $L^2(\partial\Omega)$, it can be represented by its standard matrix $\mathsf{K}$. By expanding $f$ in the orthonormal basis $\{ \phi_1, \phi_2, \ldots, \phi_n \}$ for $F_h$, $f(x)=\sum_i f_i \, \phi_i(x)$, we obtain the Euclidean approximation of \eqref{eq:pre3}: 
\begin{equation} \label{eq:pre4}
    \min_{\mathbf{f}\in\mathbb{R}^n}\left\{\frac{1}{2}\|\mathsf{M}^\frac{1}{2}_\partial \tilde{\mathsf{K}}\mathbf{f}-\mathsf{M}^\frac{1}{2}_\partial \mathbf{d}\|_2^2 + \alpha\sum_i w_i|f_i|\right\},
\end{equation}
where $\mathbf{f}$ and $\mathbf{d}$ are the Euclidean vectors associated with $f$ and $d$, respectively, $\mathsf{M}_\partial$ represents the (so-called) boundary mass matrix, and $\tilde{\mathsf{K}}$ is a numerical approximation of $\mathsf{K}$ obtained from a discretization of \eqref{eq:pre1}, see Appendix \ref{app:discretization}. 
}

\section{Motivation}
\label{sec:motivation}
The purpose of EEG is to recover electrical activity in the brain from voltage recordings on the scalp. If one suspects that the true signal is spatially local, e.g., for focal epileptic seizures, it is natural to search for sparse solutions.

Similarly, in the inverse problem of ECG, the aim can be to locate an ischemic\footnote{Ischemia is a precursor of a heart infarct.} region of the heart. This area will have an electrical potential which is different from the voltage in healthy tissue. The difference in the potential can be interpreted as the source in an elliptic PDE. If we assume that the ischemic region is small, with a sharp transition between ischemic and healthy tissue, it is reasonable to search for a sparse inverse solution.

Solving inverse source problems using optimization procedures is challenging. For example, \textcolor{black}{now employing {\em standard Tikhonov regularization} instead of the weighted sparsity approach}, the minimizer $f_\alpha \in F_h \subset L^2(\Omega)$ of 
    \begin{equation*}
        \min_{(f,u) \in F_h \times H^1(\Omega)} \left\{ \frac{1}{2}\|u - d\|_{L^2(\partial\Omega)}^2 + \frac{1}{2}\alpha\|f\|_{L^2(\Omega)}^2 \right\}
    \end{equation*}
    subject to 
    \begin{equation*}
    \begin{split}
        -\Delta u + \epsilon u &= f \quad \mbox{in } \Omega, \\
        \frac{\partial u}{\partial \nn} &= 0  \quad \mbox{on } \partial \Omega,
    \end{split}
    \end{equation*}
is not, in general, a good approximation of a true interior source: Even if the data $d$ is generated from a single basis function $\phi_j \in F_h$, representing an interior local source, the minimum $L^2$-norm least-squares solution $f^* = \lim_{\alpha\rightarrow0} f_{\alpha}$ will attain its maximum at the boundary $\partial \Omega$ of the domain $\Omega$, see \cite{Elv20}. However, by employing a carefully chosen "diagonal" regularization operator $W$, i.e., $$W \phi_i = w_i \phi_i, \quad i=1,2,\ldots,n,$$ with suitable weights $\{ w_i \}$, we proved in \cite{Elv20} that $W^{-1} f^*$ attains its maximum at the position of the true local source $\phi_j$. On the other hand, $W^{-1} f^*$ is very smooth and its magnitude will typically not be close to the magnitude of the true source.

The need for sparse solutions in many applications, as exemplified above, thus motivates the study of $\ell^1$-regularization for inverse source problems. We now present a brief overview of the main results of our forthcoming analysis.

Assume that the Dirichlet boundary data $d \in L^2(\partial\Omega)$ in \eqref{eq1} is generated from a basis function $\phi_j \in F_h$, i.e., $$d = \KK\phi_j,$$ where $\KK$
is the forward operator \eqref{eq:pre2} associated with \eqref{eq1}-\eqref{eq2}. Provided that $W$ denotes the above-mentioned regularization operator introduced in \cite{Elv20}, we will prove that $\phi_j$ is the unique solution of
\begin{equation*}
    \min_{f \in F_h} \sum_i w_i |(f, \phi_i)_{L^2(\Omega)}| \quad \textnormal{subject to} \quad \KK f = \KK \phi_j.
\end{equation*}
We will further show, using a slightly modified fidelity term in \eqref{eq1}-\eqref{eq2}, that the associated minimizer of the weighted $\ell^1$-regularized problem is 
\begin{equation*}
 f_\alpha = \gamma_\alpha \phi_j,
\end{equation*}
where $\gamma_\alpha = 1 - c\alpha$, and $c$ is a positive constant. That is, the correct basis function is recovered without any blurring for all values of $\alpha < \bar{\alpha}$, albeit with an error in magnitude equal to $c\alpha$. The constants $\bar{\alpha}$ and $c$ can be computed from $W$ and a projection operator. 

Ideally, we would like to analyze the recovery of more general composite sources. We have not been able to do so, but we address this issue numerically in Section \ref{sec:numerical_experiments}. 

\section{Analysis}\label{sec:analysis}
\textcolor{black}{As mentioned earlier}, we will present our theoretical results in terms of Euclidean spaces, employing the standard inner product and the standard basis vectors. It should be noted that analogous results can be established for linear operators acting on finite dimensional vector spaces, provided that an orthonormal basis is employed for the domain of the operators. 

Consider the problem 
\begin{equation} \label{eq:BF1}
      \min_{\xx\in\Rn} \left\{ \frac{1}{2}\|\AAA\xx - \bb\|_2^2 + \alpha\|\WW\xx\|_1 \right\},
   \end{equation}
where $\AAA \in \Rmn$ has a non-trivial null space\footnote{Our analysis also holds for matrices with linearly independent columns, but this is the trivial case.} $\nullspace{\AAA}$. Associated with $\AAA$ is the orthogonal projection matrix \begin{equation}\label{eq:Pmatrix}
\PP: \Rn \rightarrow \mathcal{N}(\AAA)^\perp.
\end{equation}
That is, $\PP = \AAA^\dagger\AAA$, where $\AAA^\dagger$ denotes the Moore-Penrose inverse of $\AAA$. 
\textcolor{black}{Note that we can write \eqref{eq:pre4} in the form \eqref{eq:BF1} by putting $\AAA = \mathsf{M}^\frac{1}{2}_\partial \tilde{\mathsf{K}}$, $\bb = \mathsf{M}^\frac{1}{2}_\partial \mathbf{d}$, $\WW = \textnormal{diag}(w_1, w_2, \ldots, w_n)$ and $\xx = \mathbf{f}$.}

Throughout this paper, the diagonal regularization matrix $\WW \in \Rnn$ is defined as
\begin{equation}\label{eq:Wmatrix}
    \WW\ee_i = \|\PP\ee_i\|_2 \ee_i, \, i=1,2,\ldots,n,
\end{equation}
where we assume that $\PP\ee_i \neq 0, \, i=1,2,\ldots,n.$ The definition of this matrix can be motivated by the classical theory for the minimum norm least squares solution of linear systems and Tikhonov regularization, see \cite{Elv20} for further details.  
Furthermore, the beneficial mathematical properties of this operator in connection with quadratic regularization are studied in \cite{Elv20,Elv20a2}. Below it will become clear that \eqref{eq:Wmatrix} also plays an important role for developing sparsity promoting regularization techniques. 

It can be CPU demanding to compute $\PP\ee_i$ for large systems because $\PP$ involves the Moore-Penrose inverse $\AAA^\dagger$ of $\AAA$. On the other hand, if the underlying problem is ill posed, such as \eqref{eq1}-\eqref{eq2}, then one would typically not use a very fine mesh for the discretization of the source, and the regularization matrix $\WW$ is applicable. Below we will approximate $\AAA^\dagger$ using either truncated SVD (Subsection \ref{subsec:altOpt}) or standard Tikhonov regularization (Subsection \ref{subsec:largeCircularSource}). 

The main purpose of this section is to analyze whether \textcolor{black}{the use of our weighted regularization technique enables the recovery of a standard basis vector $\ee_j \in \Rn$ from the exact data $\bb^\dagger = \AAA\ee_j$.}
Our starting point is thus the following optimization problem.
\begin{itemize}
    \item \textbf{Problem 0:} 
    \begin{equation}\label{eq:genmin}
      \min_{\xx\in\Rn} \left\{ \frac{1}{2}\|\AAA\xx - \bb^\dagger\|_2^2 + \alpha\|\WW\xx\|_1 \right\}.
   \end{equation}
\end{itemize}

\subsection{Weighted basis pursuit}

Our first result concerns the solution of Problem 0 in the limit $\alpha \rightarrow 0$. 
Recalling that $\bb^\dagger = \AAA\ee_j$, it is well-known that this limit problem can be formulated as
\begin{itemize}
    \item \textbf{Problem I:} 
    \begin{equation*}
      \min_{\xx\in\Rn} \|\WW\xx\|_1 \quad \textnormal{subject to} \quad \AAA\xx = \AAA\ee_j.
   \end{equation*}
\end{itemize}
We also introduce an equivalent formulation of Problem I, which reads 
 \begin{itemize}
  \item \textbf{Problem II:}
    \begin{equation*}
      \min_{\xx\in\Rn} \|\WW\xx\|_1 \quad \textnormal{subject to} \quad \PP\xx = \PP\ee_j.
  \end{equation*}
  \end{itemize}
Problems I and II are equivalent because the null spaces of $\AAA$ and $\PP$ coincide, i.e., $\mathcal{N}(\AAA) = \mathcal{N}(\PP)$. Hence, either problem can be reformulated as 
  \begin{equation*}
    \min_{\mathbf{q} \in \mathcal{N}(\AAA)} \|\WW(\ee_j + \mathbf{q})\|_1.
  \end{equation*}
  
 We will assume that no {\em two} columns of $\AAA$ are parallel in order to ensure that the solutions of our minimization problems are {\em unique}.  
\begin{assumption}\label{nonpar}
 We assume that $\AAA \in \Rmn$ is such that $\AAA\ee_j \neq c\AAA\ee_i$ for all $i, j \in \{1,2,...,n\}$, $i \neq j$, and all $c \in \mathbb{R}$.
\end{assumption}
\noindent Note that this assumption implies that $\ee_i \notin \mathcal{N}(A)$, \, $i \in \{1,2,...,n\}$. Furthermore, invoking the orthogonal decomposition $\ee_i = \PP \ee_i + (\ee_i - \PP \ee_i)$, where $(\ee_i - \PP \ee_i) \in \nullspace{A}$, it follows that $\AAA \ee_i = \AAA \PP \ee_i$. Consequently, if $\AAA$ obeys Assumption \ref{nonpar}, 
then $\PP$ must also satisfy 
\begin{equation}
    \label{Pnonpar}
    \PP \ee_j \neq c\PP \ee_i \mbox{ for all } i, j \in \{1,2,...,n\}, \, i \neq j, \mbox{ and all } c \in \mathbb{R}. 
\end{equation}

We can now formulate our first result. Let us remark that the study of sparsity promoting regularization techniques usually leads to very involved mathematical analysis. However, with the very specific choice of the weight matrix $\WW$ in \eqref{eq:Wmatrix} and the data $\bb^\dagger = \AAA\ee_j$, our analysis becomes rather "transparent" \textcolor{black}{because we only consider the recovery of a $1$-sparse solution}. 
\begin{theorem}[Exact recovery of a basis vector]\label{thm:exactRecov}
    Assume that $\AAA \in \Rmn$ satisfies Assumption \ref{nonpar}, and let $\PP$ and $\WW$ be the matrices defined in \eqref{eq:Pmatrix} and \eqref{eq:Wmatrix}, respectively. Then $\ee_j$ is the unique solution of problems I and II.
\end{theorem}
\begin{proof}
  Problem I and Problem II are equivalent. We choose to consider the latter, i.e.,
    \begin{equation*}
      \min_{\xx\in\Rn} \|\WW\xx\|_1 \quad \textnormal{subject to} \quad \PP\xx = \PP\ee_j.
  \end{equation*}
  Define $X_j = \{\xx\in\Rn: \PP\xx = \PP\ee_j\}$. Let $\xx = \sum_i c_i\ee_i\in X_j, \xx \neq \ee_j$ be arbitrary. Then, see \eqref{eq:Wmatrix},
  \begin{eqnarray*}
    \|\WW\ee_j\|_1 &=& \left\| \|\PP\ee_j\|_2 \ee_j \right\|_1 \\
    &=& \|\PP\ee_j\|_2 \\
    &=& \left\|\PP\left(\sum_i c_i\ee_i\right)\right\|_2 \\
    &=& \left\|\sum_i c_i\PP\ee_i\right\|_2 \\
    &\leq& \sum_i |c_i| \|\PP\ee_i\|_2 \\
    &=& \left\| \sum_i c_i \|\PP\ee_i\|_2 \ee_i \right\|_1\\
    &=& \|\WW\xx\|_1.
  \end{eqnarray*}
Using \eqref{Pnonpar}, which is a consequence of Assumption \ref{nonpar}, we get strict inequality in the third to last step, and the result follows.
\end{proof}
\noindent Theorem \ref{thm:exactRecov} states that a single basis vector $\ee_j$ can be exactly recovered from the data $\bb^\dagger = \AAA\ee_j$ by solving either Problem I or Problem II. If Assumption \ref{nonpar} does not hold, $\ee_j$ would still be a minimizer, but the uniqueness is not assured. (Similar statements hold for the remaining results presented in this paper.)

\subsection{Alternative optimization problems}
\label{subsec:altOpt}
Inspired by Problem II, we now suggest an alternative to \eqref{eq:BF1}. 
Since $\PP = \AAA^\dagger\AAA$, it follows that 
\[
\PP \ee_j = \AAA^\dagger\AAA\ee_j =  \AAA^\dagger \bb^\dagger,
\]
where $\bb^\dagger = \AAA\ee_j$.
Consequently, for a general right-hand-side $\bb$, Problem II motivates the following alternative to \eqref{eq:BF1} 
\begin{equation}
    \label{eq:BF2}
    \min_{\xx\in\Rn} \left\{ \frac{1}{2} \|\PP\xx - \AAA^\dagger \bb\|_2^2 + \alpha\|\WW\xx\|_1 \right\}.
\end{equation}
If $\AAA$ has very small singular values, it may not be advisable to apply $\AAA^\dagger$ in practise. Therefore we want to approximate  $\AAA^\dagger$ with a more well-behaved matrix. This can, e.g., be accomplished by employing the truncated SVD to get an approximation $\AAK$ of $\AAA$. Here, $k \leq m$ represents the number of singular values that are unchanged in the truncation: Assuming that the singular values of $\AAA$ are sorted in decreasing order $\sigma_1 \geq \sigma_2 \geq \ldots \geq \sigma_m$, $\AAK$ will have the singular values $\sigma_1 \geq \sigma_2 \geq \ldots \geq \sigma_k \geq 0 = 0 = \ldots = 0$. 

Below we need the orthogonal projection onto the orthogonal complement of the null space of $\AAK$, 
\begin{equation}
\label{eq:defPK}
\PPK: \mathbb{R}^n \rightarrow \nullspace{\AAK}^\perp. 
\end{equation}
And, analogously to \eqref{eq:Wmatrix}, we define 
\begin{equation}\label{eq:WKmatrix}
    \WWK \ee_i = \|\PPK \ee_i\|_2 \ee_i, \, i=1,2,\ldots,n.
\end{equation}
Replacing $\AAA^\dagger$ in \eqref{eq:BF2} with $\AAK^\dagger$, keeping in mind that $\PP = \AAA^\dagger\AAA$, leads to the following alternative to \eqref{eq:BF1} 
      \begin{equation} \label{eq:BF3}
        \min_{\xx\in\Rn} \left\{ \frac{1}{2} \|\AAK^\dagger \AAA\xx - \AAK^\dagger \bb\|_2^2 + \alpha\|\WWK\xx\|_1 \right\}.
      \end{equation}
Note that we have also replaced $\WW$ with $\WWK$. 

From the SVD of $\AAA=\UU \SI \VV^T$, we find the SVD of $\AAK$, 
\[
\AAK = \UU \SIK \VV^T. 
\]
Observe that 
\[
\PPK=\AAK^\dagger \AAK = \VV \SIK^\dagger \SIK \VV^T, 
\]
and that 
\[
\AAK^\dagger \AAA = \VV \SIK^\dagger \SI \VV^T. 
\]
This yields, since $\SIK^\dagger \SIK = \SIK^\dagger \SI$,
\begin{equation}
\label{eq:projectionProperty}
    \PPK=\AAK^\dagger \AAA, 
\end{equation}
and we can write \eqref{eq:BF3} in the form 
\begin{equation}
\label{eq:BF4}
        \min_{\xx\in\Rn} \left\{ \frac{1}{2} \|\PPK \xx - \AAK^\dagger \bb\|_2^2 + \alpha\|\WWK\xx\|_1 \right\}.
\end{equation}

In the next subsections we will analyze \eqref{eq:BF2} and \eqref{eq:BF4} when $\bb = \AAA\ee_j$ and $\bb = \AAA\ee_j + \eta$, respectively, where $\eta$ represents noise.

\subsection{Analysis of regularized problems} 
\textcolor{black}{
 In this subsection we will make use of the following maximum property derived in \cite{Elv20} and \cite{Elv20a2}:
 \begin{equation}\label{eq:maxprop}
     j = \argmax_{i \in \{1,2,...,n\}} |[\WW^{-1}\PP\ee_j]_i|, 
 \end{equation}
 where $[\WW^{-1}\PP\ee_j]_i$ denotes the $i$'th component of the vector $\WW^{-1}\PP\ee_j$, and $\PP$ and $\WW$ are defined in \eqref{eq:Pmatrix} and \eqref{eq:Wmatrix}, respectively. 
 More precisely, the proof of Theorem 4.2 in \cite{Elv20} reveals that 
 \begin{eqnarray}
  \WW^{-1}\PP\ee_j 
   \label{eq:revision1}
  = \|\PP\ee_j\|\sum_{i=1}^n\left(\frac{\PP\ee_j}{\|\PP\ee_j\|}, \frac{\PP\ee_i}{\|\PP\ee_i\|}\right)\ee_i,
 \end{eqnarray}
which combined with Assumption \ref{nonpar} yields \eqref{eq:maxprop}, see \cite{Elv20a2} for further details. 
}
 
With $\bb = \AAA \ee_j$, \eqref{eq:BF2} reads 
\begin{itemize}
    \item \textbf{Problem III:}
      \begin{equation} \label{eq:PPmin}
        \min_{\xx\in\Rn} \left\{ \frac{1}{2} \|\PP\xx - \PP\ee_j\|_2^2 + \alpha\|\WW\xx\|_1 \right\}.
      \end{equation}
\end{itemize}
\noindent We will now see that the maximum property \eqref{eq:maxprop} allows an analysis of this problem which only involves   classical convex optimization theory.

\begin{theorem}\label{thm:Palpha}
    Assume that the matrix $\AAA \in \Rmn$ satisfies Assumption \ref{nonpar}, and let $\PP$ and $\WW$ be the matrices defined in \eqref{eq:Pmatrix} and \eqref{eq:Wmatrix}, respectively. Then $$\xx^*_\alpha = \gam \ee_j$$ is the unique solution of Problem III, where
    \begin{equation}\label{eq:gamma}
        \gam = 1 - \frac{\alpha}{[\WW^{-1}\PP\ee_j]_j} \quad \mbox{for } 0 < \alpha < [\WW^{-1}\PP\ee_j]_j.
    \end{equation}
\end{theorem}
\begin{proof} $ $ \newline 
  \textit{Existence:} Let us define the cost-functional $\mathcal{J}: \Rn \rightarrow \mathbb{R}$ associated with \eqref{eq:PPmin}, 
  \begin{equation} \label{eq:def_J_g_h}
      \mathcal{J}(\xx) = \underbrace{\frac{1}{2}\|\PP\xx-\PP\ee_j\|_2^2}_{=g(\xx)}
      + \underbrace{\alpha\|\WW\xx\|_1}_{=\alpha h(\WW\xx)},
  \end{equation}
  where $g(\cdot)$ and $h(\WW\cdot)$ represent the fidelity and regularization terms, respectively.
  According to standard convex optimization theory, $\xx$ is a minimizer of $\mathcal{J}$ if and only if
  \begin{eqnarray*}
      \mathbf{0} &\in& \partial\mathcal{J}(\xx) \\
      &=& \nabla g(\xx) + \alpha\WW^T\partial h(\WW\xx), 
  \end{eqnarray*}
  where "$\partial$" denotes the subgradient. 
  Since $\WW^T = \WW$, we can multiply with $\WW^{-1}$ to obtain
  \begin{equation*}
     -\WW^{-1}\nabla g(\xx) \in \alpha\partial h(\WW\xx),
  \end{equation*}
  and from the expression for $g$ we find, keeping in mind that $\PP^T \PP = \PP \PP = \PP$, 
  \begin{equation}\label{eq:gradG}
     \WW^{-1}\PP(\ee_j - \xx) \in \alpha \partial h(\WW\xx).
  \end{equation}  
  We also observe, using the fact that $h(\yy)=\| \yy \|_1$ and that $\WW$ is a diagonal matrix with positive entries at its diagonal,  
  \begin{equation*}
     [\partial h(\WW\xx)]_i =  [\partial h(\WW [x_1 \, x_2 \, \ldots \, x_n]^T)]_i =  
        \begin{cases} 
            \{1\}, & x_i > 0, \\
            \{-1\}, & x_i < 0, \\
            [-1,1], & x_i = 0.
        \end{cases}
  \end{equation*}
  
  We will now investigate whether there exists a scalar $\gamma$ such that $\xx = \gamma\ee_j$ satisfies the optimality criterion  \eqref{eq:gradG}. Note that, for $\gamma >0$,  
  \begin{equation} \label{eq:BFslow2}
  [\partial h(\WW\gamma \ee_j)]_i =         
        \begin{cases} 
            \{1\}, & i = j, \\
            [-1,1], & i \neq j,
        \end{cases}
  \end{equation}
  and the condition \eqref{eq:gradG}, with $\xx = \gamma\ee_j$, becomes   
  \begin{equation} \label{eq:BFneedsThisToUnderstand}
      (1-\gamma)[\WW^{-1}\PP\ee_j]_i \in \alpha
        \begin{cases} 
            \{1\}, & i = j, \\
            [-1,1], & i \neq j.
        \end{cases}
  \end{equation}
  Setting
  \begin{equation}
  \label{eq:gammaDef}
      \gamma = \gam = 1 - \frac{\alpha}{[\WW^{-1}\PP\ee_j]_j},
  \end{equation}
  we observe from \eqref{eq:maxprop} that
    \begin{equation}\label{eq:uniquecond}
      (1-\gam)[\WW^{-1}\PP\ee_j]_i = \alpha \, \frac{[\WW^{-1}\PP\ee_j]_i}{[\WW^{-1}\PP\ee_j]_j}  \in \alpha
        \begin{cases} 
            \{1\}, & i = j, \\
            (-1,1), & i \neq j, 
        \end{cases}
  \end{equation}
  and we conclude that \eqref{eq:BFneedsThisToUnderstand} holds for the particular choice \eqref{eq:gammaDef} of $\gamma$. 
  
  This argument shows that $\xx_\alpha^* = \gam\ee_j$ is a minimizer of $\mathcal{J}$. The next step is to use the property that $(1-\gam)[\WW^{-1}\PP\ee_j]_i$ is contained in the \underline{open} interval $(-\alpha,\alpha)$, $i \neq j$, to prove the uniqueness. 
  \newline \newline \noindent 
  \textit{Uniqueness:} 
  We have determined a minimizer $\xx_\alpha^* = \gam\ee_j$ of $\mathcal{J}$. Let $\mathbf{y} \in \Rn, \mathbf{y} \neq \xx_\alpha^*$, be arbitrary. We will show that
  \begin{equation*}
      \mathcal{J}(\yy) > \mathcal{J}(\xx_\alpha^*).
  \end{equation*}
  If $\mathbf{y} = c\xx_\alpha^*$, $c \neq 1$, it follows from the analysis presented above that this is not a minimizer of $\mathcal{J}$. Consequently, for the remaining part of the proof we assume that $\mathbf{y} \neq c\xx_\alpha^*$. In particular, this implies that at least one of the components, say $y_k$, $k \neq j$, of $\mathbf{y}$ is such that $y_k \neq 0$.
 
 Recall the definition \eqref{eq:def_J_g_h} of $\mathcal{J}$, $g$ and $h$ and that, using the definition of the subgradient,  
 \[
 h(\WW \yy) - h(\WW \xx_\alpha^*) \geq \mathbf{z}^T(\WW\yy - \WW\xx_\alpha^*) \quad  \mbox{for all }  \mathbf{z} \in \partial h(\WW\xx_\alpha^*). 
 \]
 Therefore 
 \begin{align}
 \nonumber
    \mathcal{J}(\mathbf{y}) - \mathcal{J}(\xx_\alpha^*) 
    &= g(\mathbf{y}) + \alpha h(\WW\mathbf{y}) - g(\xx_\alpha^*) - \alpha h(\WW\xx_\alpha^*) \\
    \nonumber
    &\geq\frac{1}{2}\|\PP\mathbf{y}-\PP\ee_j\|_2^2 - \frac{1}{2}\|\PP\xx_\alpha^*-\PP\ee_j\|_2^2 \\
    \label{eq:BFslow1}
    &\quad +\alpha \mathbf{z}^T \WW (\yy - \xx_\alpha^*) \quad  \mbox{for all }  \mathbf{z} \in \partial h(\WW\xx_\alpha^*).
  \end{align}
 Since $\xx_\alpha^* = \gam\ee_j$, we can write \eqref{eq:uniquecond} as
    \begin{align}
    \label{eq:BFslow3}
      [\WW^{-1}\PP(\ee_j-\xx_\alpha^*)]_i &\in \alpha
        \begin{cases} 
            \{1\}, & i = j \\
            (-1,1), & i \neq j 
        \end{cases} \\
        \nonumber
        &\subset \alpha
        \begin{cases} 
            \{1\}, & i = j \\
            [-1,1], & i \neq j 
        \end{cases} \\
        \nonumber
        &= \alpha [\partial h(\WW\xx_\alpha^*)]_i,
  \end{align}  
  see \eqref{eq:BFslow2}, i.e., 
  \begin{equation}\label{eq:insubdiff}
      \frac{1}{\alpha}\WW^{-1}\PP(\ee_j - \xx_\alpha^*) \in \partial h(\WW\xx_\alpha^*).
  \end{equation}
  Hence, we could choose $\mathbf{z} = \frac{1}{\alpha}\WW^{-1}\PP(\ee_j - \xx_\alpha^*)$ in \eqref{eq:BFslow1}, but then we do not (directly) get a strict inequality. 
  Recall that $\mathbf{y}$ has a component $y_k \neq 0$, where $k \neq j$. Without loss of generality, we may assume that $[\WW\mathbf{y}-\WW\xx_\alpha^*]_k > 0$. 
  Define $\mathbf{\tilde{z}}=[\tilde{z}_1 \, \tilde{z}_2 \, \ldots \tilde{z}_n]^T$ as follows\footnote{If $[\WW\mathbf{y}-\WW\xx_\alpha^*]_k < 0$, define $\tilde{z}_k = -1$, etc.} 
  \begin{equation*}
      \tilde{z}_i = 
      \begin{cases} 
            1, & i = k, \\
            \frac{1}{\alpha}[\WW^{-1}\PP(\ee_j-\xx_\alpha^*)]_i, & i \neq k.  
        \end{cases} 
  \end{equation*}
Since \eqref{eq:BFslow3} implies that $\left|\frac{1}{\alpha}[\WW^{-1}\PP(\ee_j-\xx_\alpha^*)]_k\right| < 1$, we find that
    \begin{equation*}
        \mathbf{\tilde{z}}^T [\WW\mathbf{y}-\WW\xx_\alpha^*] > \frac{1}{\alpha}[\WW^{-1}\PP(\ee_j-\xx_\alpha^*)]^T [\WW\mathbf{y}-\WW\xx_\alpha^*].
    \end{equation*}
Due to \eqref{eq:BFslow3} and \eqref{eq:BFslow2}, $\mathbf{\tilde{z}} \in \partial h(\WW\xx_\alpha^*)$, and therefore \eqref{eq:BFslow1} yields  
 \begin{align*}
    \mathcal{J}(\mathbf{y}) - \mathcal{J}(\xx_\alpha^*) 
    &\geq\frac{1}{2}\|\PP\mathbf{y}-\PP\ee_j\|_2^2 - \frac{1}{2}\|\PP\xx_\alpha^*-\PP\ee_j\|_2^2 \\
    &\quad +\alpha \mathbf{\tilde{z}}^T \WW (\yy - \xx_\alpha^*) \\
    &> \frac{1}{2}\|\PP\mathbf{y}-\PP\ee_j\|_2^2 - \frac{1}{2}\|\PP\xx_\alpha^*-\PP\ee_j\|_2^2 \\
    &\quad + \alpha \frac{1}{\alpha}[\WW^{-1}\PP(\ee_j-\xx_\alpha^*)]^T [\WW\mathbf{y}-\WW\xx_\alpha^*] \\
    &= \frac{1}{2}\|\PP\mathbf{y}-\PP\ee_j\|_2^2 - \frac{1}{2}\|\PP\xx_\alpha^*-\PP\ee_j\|_2^2 \\
    &\quad + [\PP(\ee_j-\xx_\alpha^*)]^T [\mathbf{y}-\xx_\alpha^*].
  \end{align*}
The gradient of $g$, see \eqref{eq:def_J_g_h}, is $\nabla g(x)=\PP (\xx-\ee_j)$. Consequently, the convexity of $g$ implies that 
\begin{align*}
    \mathcal{J}(\mathbf{y}) - \mathcal{J}(\xx_\alpha^*) 
    &> \frac{1}{2}\|\PP\mathbf{y}-\PP\ee_j\|_2^2 - \frac{1}{2}\|\PP\xx_\alpha^*-\PP\ee_j\|_2^2 \\
    &\quad - \nabla g(\xx_\alpha^*)^T [\mathbf{y}-\xx_\alpha^*] \\
    &\geq \frac{1}{2}\|\PP\mathbf{y}-\PP\ee_j\|_2^2 - \frac{1}{2}\|\PP\xx_\alpha^*-\PP\ee_j\|_2^2 \\
    &\quad - [g(\mathbf{y})-g(\xx_\alpha^*)] \\
    &=\frac{1}{2}\|\PP\mathbf{y}-\PP\ee_j\|_2^2 - \frac{1}{2}\|\PP\xx_\alpha^*-\PP\ee_j\|_2^2 \\ 
    &\quad - \left[ \frac{1}{2}\|\PP\mathbf{y}-\PP\ee_j\|_2^2 - \frac{1}{2}\|\PP\xx_\alpha^*-\PP\ee_j\|_2^2 \right] \\ 
    &=0, 
  \end{align*}
  which finishes the proof. 
\end{proof}

This theorem shows that $\xx_\alpha^* = \gam\ee_j$ is the unique minimizer of \eqref{eq:PPmin}.
The solution of \eqref{eq:PPmin} is thus obtained by only changing the magnitude of the true source $\ee_j$, where the scaling factor $\gam \rightarrow 1$ as $\alpha \rightarrow 0$. 

\subsubsection{Noisy observation data}
As mentioned in Subsection \ref{subsec:altOpt}, it may not be advisable to apply the pseudo-inverse $\AAA^\dagger$ in practical computations. We therefore now want to study \eqref{eq:BF3} in more detail. Setting $\bb=\AAA \ee_j + \eta$ in \eqref{eq:BF4}, where $\eta \in \mathbb{R}^m$ represents noise, leads to 
\[
\min_{\xx\in\Rn} \left\{ \frac{1}{2} \|\PPK \xx - (\AAK^\dagger \AAA\ee_j + \AAK^\dagger \eta)\|_2^2 + \alpha\|\WWK \xx\|_1 \right\}. 
\]
Since $\AAK^\dagger \AAA = \PPK$, see \eqref{eq:projectionProperty}, this problem can also can be written in the form 
\begin{itemize}
    \item \textbf{Problem IV:} 
      \begin{equation} \label{eq:PPminnoise}
        \min_{\xx\in\Rn} \left\{ \frac{1}{2} \|\PPK \xx - (\PPK \ee_j + \AAK^\dagger \eta )\|_2^2 + \alpha\|\WWK \xx\|_1 \right\}.
      \end{equation}
\end{itemize}

\begin{theorem}\label{thm:Pnoise}
    Assume that $\AAK \in \Rmn$ satisfies Assumption \ref{nonpar}, and let $\PPK$ and $\WWK$ be the matrices defined in \eqref{eq:defPK} and \eqref{eq:WKmatrix}, respectively. Then $$\xx^*_{\alpha,\eta} = \gamn \ee_j$$ is the unique solution of Problem IV, where 
    \begin{equation}\label{eq:gammanoise}
        \gamn = 1 - \frac{\alpha + [\WWK^{-1}\AAK^\dagger\eta]_j}{[\WWK^{-1}\PPK\ee_j]_j}.
    \end{equation}    
    In order for this to hold, $\alpha$ must obey
  \begin{equation}
  \label{ineq:bounds}
      \max_{i \neq j} \frac{1 + |\tau_{ij}|}{1-|\tau_{ij}|} \max_i \left|[\WWK^{-1}\AAK^\dagger\eta]_i\right| < \alpha < [\WWK^{-1}\PPK\ee_j]_j - [\WWK^{-1}\AAK^\dagger\eta]_j,
  \end{equation}
  where $$\tau_{ij} = \frac{[\WWK^{-1}\PPK\ee_j]_i}{[\WWK^{-1}\PPK\ee_j]_j} \in \left(-1,1\right), \quad i\neq j.$$
\end{theorem}
\begin{proof}
  Following the same reasoning as in the proof of Theorem \ref{thm:Palpha}, we derive that $x_\alpha^* = \gamma\ee_j$ is the unique minimizer of  \eqref{eq:PPminnoise} if 
  \begin{equation}\label{eq:condnoise}
        (1-\gamma)[\WWK^{-1}\PPK\ee_j]_i - [\WWK^{-1}\AAK^\dagger\eta]_i \in \alpha
        \begin{cases} 
            \{1\}, & i = j \\
            (-1,1), & i \neq j,
        \end{cases}
  \end{equation}
  where we have used the fact that $\PPK^T=\PPK=\AAK^\dagger \AAK$, and consequently  
  \[
  \PPK^T \AAK^\dagger = \AAK^\dagger \AAK \AAK^\dagger = \AAK^\dagger. 
  \]
  The criterion \eqref{eq:condnoise} holds for $i = j$ if
  \begin{equation*}
      \gamma = \gamn = 1 - \frac{\alpha + [\WWK^{-1}\AAK^\dagger\eta]_j}{[\WWK^{-1}\PPK\ee_j]_j} > 0.
  \end{equation*}
  Consequently, $\alpha$ must satisfy the upper bound
  \begin{equation*}
      \alpha < [\WWK^{-1}\PPK\ee_j]_j - [\WWK^{-1}\AAK^\dagger\eta]_j.
  \end{equation*}
  Furthermore, setting $\gamma = \gamn$ in \eqref{eq:condnoise}, the condition \eqref{eq:condnoise} for $i \neq j$ reads
  \begin{equation}
  \label{eq:BFextraStep}
      \left(\alpha + [\WWK^{-1}\AAK^\dagger\eta]_j\right)\frac{[\WWK^{-1}\PPK\ee_j]_i}{[\WWK^{-1}\PPK\ee_j]_j}
      \in \left(-\alpha + [\WWK^{-1}\AAK^\dagger\eta]_i, \, \alpha + [\WWK^{-1}\AAK^\dagger\eta]_i\right).
  \end{equation}
  Recall that $\tau_{ij} := \frac{[\WWK^{-1}\PPK\ee_j]_i}{[\WWK^{-1}\PPK\ee_j]_j} \in (-1, \,1)$ for $i \neq j$, see\footnote{The property \eqref{eq:maxprop} holds for any matrix $\AAA$ satisfying Assumption \ref{nonpar}, where $\PP$ and $\WW$ are defined in \eqref{eq:Pmatrix} and \eqref{eq:Wmatrix}, respectively. Consequently, \eqref{eq:maxprop} also holds for $\AAK$, replacing $\WW$ with $\WWK$ and $\PP$ with $\PPK$, provided that $\AAK$ also satisfies Assumption \ref{nonpar}.} \eqref{eq:maxprop}, and \eqref{eq:BFextraStep} can thus be expressed as the two inequalities
  \begin{eqnarray}
        \left(1 + \tau_{ij}\right)\alpha &> [\WWK^{-1}\AAK^\dagger\eta]_i - \tau_{ij}[\WWK^{-1}\AAK^\dagger\eta]_j, \label{eq:cond1} \\ 
        \left(1-\tau_{ij}\right)\alpha &> \tau_{ij}[\WWK^{-1}\AAK^\dagger\eta]_j - [\WWK^{-1}\AAK^\dagger\eta]_i. \label{eq:cond1b}
  \end{eqnarray}
 It turns out that both of these inequalities are satisfied, for $i \neq j$, if 
   \begin{equation}
      \max_{i \neq j} \frac{1 + |\tau_{ij}|}{1-|\tau_{ij}|} \max_i \left|[\WWK^{-1}\AAK^\dagger\eta]_i\right| < \alpha. \label{eq:cond1_satisfied}
  \end{equation}
  
 Let us end the proof by verifying that \eqref{eq:cond1} holds if \eqref{eq:cond1_satisfied} is satisfied. Since $(1+\tau_{ij}) > 0$, the requirement \eqref{eq:cond1} can be written in the form 
 \begin{eqnarray*}
  \alpha > \frac{[\WWK^{-1}\AAK^\dagger\eta]_i - \tau_{ij}[\WWK^{-1}\AAK^\dagger\eta]_j}{1+\tau_{ij}}.
 \end{eqnarray*}
 We derive the following inequalities, considering the case $i \neq j$, 
 \begin{eqnarray*}
   \frac{[\WWK^{-1}\AAK^\dagger\eta]_i - \tau_{ij}[\WWK^{-1}\AAK^\dagger\eta]_j}{1+\tau_{ij}} 
   &\leq&   \frac{\left|[\WWK^{-1}\AAK^\dagger\eta]_i - \tau_{ij}[\WWK^{-1}\AAK^\dagger\eta]_j\right|}{\left|1+\tau_{ij}\right|} \\
   &\leq& \frac{\left|[\WWK^{-1}\AAK^\dagger\eta]_i\right| + |\tau_{ij}|\left|[\WWK^{-1}\AAK^\dagger\eta]_j\right|}{1-\left|\tau_{ij}\right|} \\
   &\leq&       \max_{i \neq j} \frac{1 + |\tau_{ij}|}{1-|\tau_{ij}|} \max_i \left|[\WWK^{-1}\AAK^\dagger\eta]_i\right|,
 \end{eqnarray*}
 and we conclude that: If $\alpha$ satisfies \eqref{eq:cond1_satisfied}, then  \eqref{eq:cond1} holds. Similarly, one verifies that \eqref{eq:cond1_satisfied} implies  \eqref{eq:cond1b}.  This finishes the proof.
\end{proof}

Roughly, the left inequality in \eqref{ineq:bounds} ensures that the error amplification caused by the inverse solution procedure does not become too dominate, and the right inequality prevents the regularization term from becoming too “strong” and thereby yielding a poor recovery of the source. 

\textcolor{black}{
We note that, in the zero-noise-limit $\| \eta \|_\infty \rightarrow 0$, the lower and upper bounds in \eqref{ineq:bounds} become $0$ and $[\WWK^{-1}\PPK\ee_j]_j$, respectively. From \eqref{eq:revision1} we find that $[\WW^{-1}\PP\ee_j]_j > 0$ and a similar argument reveals that also $[\WWK^{-1}\PPK\ee_j]_j >0$, cf. definitions \eqref{eq:Pmatrix}, \eqref{eq:Wmatrix}, \eqref{eq:defPK} and \eqref{eq:WKmatrix} of $\PP$, $\WW$, $\PPK$ and $\WWK$, respectively. Hence, provided that the noise level is sufficiently small, one can in principle always choose the size of $\alpha$ such that \eqref{ineq:bounds} holds. On the other hand, as the degree of noise increases, \eqref{ineq:bounds} may not hold for any $\alpha > 0$.     
}

\textcolor{black}{One can also use standard Tikhonov regularization to obtain an approximation of the Moore-Penrose inverse $\AAA^\dagger$ of $\AAA$. We will explore this approach numerically in Subsection \ref{subsec:largeCircularSource}. It is, however, an open problem how to modify the proof of Theorem \ref{thm:Pnoise} to Tikhonov based approximations of $\AAA^\dagger$.}

\begin{remark}[\bf{Several sources}]
 Let us mention that the methods introduced in this paper can not, in general, guarantee the recovery of multiple sources. To show this, assume that the exact data is $\bb^\dagger = \AAA\ee_m + \AAA\ee_n$ and that there exist a constant $c$ and an index $j$ such that $\AAA\ee_m + \AAA\ee_n = c\AAA\ee_j$. 
 
 Recall that $\PP = \AAA^\dagger\AAA$. Hence, multiplying $\AAA\ee_m + \AAA\ee_n = c\AAA\ee_j$ with $\AAA^\dagger$ yields
 \begin{equation*}
     \PP\ee_m + \PP\ee_n = c\PP\ee_j.
 \end{equation*}
 The weighted basis pursuit problem, cf. Problem II, then reads
  \begin{equation}
  \label{problemSS}
     \min_{\xx\in\Rn} \|\WW\xx\|_1 \quad \textnormal{subject to} \quad \PP\xx = \PP\ee_m + \PP\ee_n.
 \end{equation}
 Following the proof of Theorem \ref{thm:exactRecov}, we get, see \eqref{eq:Wmatrix},
 \begin{eqnarray*}
  \|c\WW\ee_j\|_1 &=& |c|\|\PP\ee_j\|_2 \\ 
  &=& |c| \left\|\frac{\PP\ee_m+\PP\ee_n}{c}\right\|_2 \\
  &=& \|\PP\ee_m + \PP\ee_n\|_2 \\
  &<& \|\PP\ee_m\|_2 + \|\PP\ee_n\|_2 \\
  &=& \|\WW(\ee_m + \ee_n)\|_1,
 \end{eqnarray*}
 where the strict inequality is a consequence of Assumption \ref{nonpar}, see also \eqref{Pnonpar}. 
 This argument shows that the two true sources $\ee_m$ and $\ee_n$ will not be recovered by solving \eqref{problemSS}. See Figure \ref{fig:1d} for an illustration.
 
 On the other hand, if there do not exist an index $j$ and a constant $c$ such that $\AAA\ee_m+\AAA\ee_n = c\AAA\ee_j$, it is an open problem whether the weighted basis pursuit formulation can recover a 2-sparse vector. Or, more generally, can we recover a  $s$-sparse solution if its image under $\AAA$ is not equal to the image under $\AAA$ of any $s'$-sparse solution with $s' < s$?
\end{remark}
\begin{figure}[h]
    \centering
    \begin{subfigure}[b]{0.45\linewidth}        
        \centering
        \includegraphics[width=\linewidth]{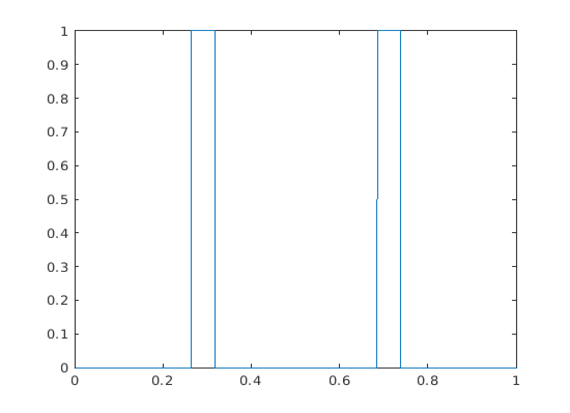}
        \caption{True source.}
    \end{subfigure}
    \begin{subfigure}[b]{0.45\linewidth}        
        \centering
        \includegraphics[width=\linewidth]{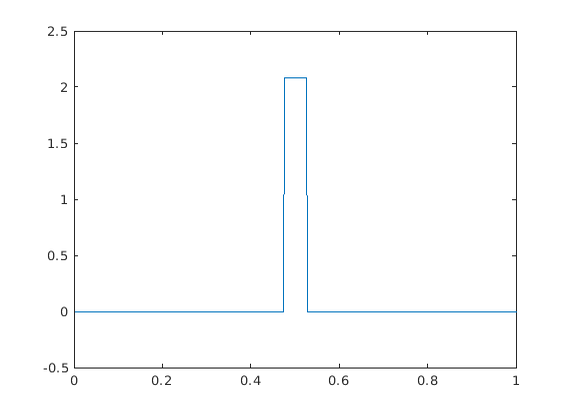}
        \caption{Inverse solution.}
    \end{subfigure}\par
    \caption{Panel (b) shows the solution $f^*$ of \eqref{eq1}-\eqref{eq2} when $\Omega = (0,1), \epsilon = 1$ and $\alpha = 0.001$. The observation data $d$ was in this case generated by the two true local sources depicted in Panel (a).}
    \label{fig:1d}
\end{figure}

\section{Numerical experiments}
\label{sec:numerical_experiments}
\textcolor{black}{
In order to illuminate our theoretical work, we committed the so-called "inverse crime" in Example 1 below: The same grid was used to both generate the boundary observation data $d$ in \eqref{eq1} and for solving the inverse problem. Consequently, the assumptions needed in Theorem \ref{thm:Palpha} are (in principle\footnote{\textcolor{black}{Disregarding round-off errors.}}) satisfied, provided that we consider the single-source-case.}  
In all the other experiments, the data $d$ was generated using a finer grid for the state $u$ than was used in the inverse computations. More specifically, $h_\textnormal{forward} = 0.5h_\textnormal{inverse}$, and we performed experiments on the unit square with $129\times129$ and $65\times65$ nodes for the forward and inverse computations of the state $u$, respectively. The source \textcolor{black}{$f \in F_h$} was discretized in terms of a $16\times16$ mesh in all the simulations presented in Examples 1-2 and 4. In Example 3, however, the true source was discretized using a finer $129\times129$ grid. Note that we employed the basis functions 
\begin{equation} \label{eq:revision2}
\pphi_i = \frac{1}{\| \mathcal{X}_{\Omega_i}  \|_{L^2(\Omega)}} \mathcal{X}_{\Omega_i}, 
\quad i=1,2, \ldots, n,  
\end{equation}
for $F_h$, where $\Omega_1, \, \Omega_2, \ldots, \Omega_n$ are uniformly sized disjoint grid cells and $\mathcal{X}_{\Omega_i}$ denotes the characteristic function of $\Omega_i$.


We employed the FEniCS software\textcolor{black}{, discretizing the state $u$ in terms of first order Lagrange elements,} to generate the matrices involved in our experiments. Thereafter, the matrices were exported to MATLAB, where the optimization problems were solved with the split-Bregman algorithm \cite{goldstein09}. Some details about the forward/transfer matrix $\AAA$ is presented in \textcolor{black}{Section \ref{sec:prelim}} and Appendix \ref{app:discretization}. We do not present a detailed description of the well-known mappings between the finite element spaces arising from the discretization of \eqref{eq1}-\eqref{eq2} to the Euclidean spaces used in \eqref{eq:BF1}, \eqref{eq:BF2} and \eqref{eq:BF4}. Note, however, that $\ee_j \in \Rn$ is associated with the FE basis function $\phi_j \in F_h \subset L^2(\Omega)$.

In all the simulations $\epsilon = 1$, see \eqref{eq2}, and no noise was added to the data $d$, except in Example 2.

Figure \ref{fig:weights} contains visualizations of the entries of the weight matrix $\WWK$ defined in \eqref{eq:WKmatrix}. More specifically, each panel shows a plot of $\|\PPK\ee_i\|, \, i=1,2,\ldots,n$. We observe that the weights are largest for indexes associated with basis functions positioned close to the boundary of the domain $\Omega$.

\begin{figure}[h]
    \centering
    \begin{subfigure}[b]{0.45\linewidth}        
        \centering
        \includegraphics[width=\linewidth]{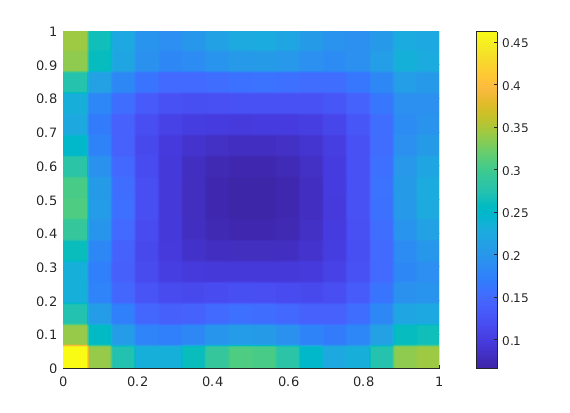}
        \caption{$k = 7$.}
    \end{subfigure}
    \begin{subfigure}[b]{0.45\linewidth}        
        \centering
        \includegraphics[width=\linewidth]{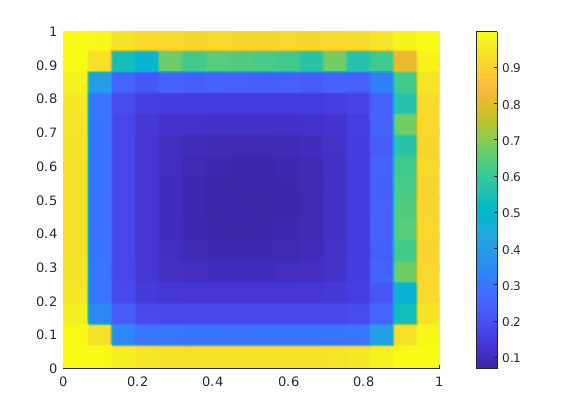}
        \caption{$k = 70$.}
    \end{subfigure}\par
    \caption{Visualizations of the matrix $\WWK$. Panels (a) and (b) show plots of $\|\PPK\ee_i\|, \, i=1,2,\ldots,n$, see \eqref{eq:WKmatrix}, for two different choices of the truncation parameter $k$.}
    \label{fig:weights}
\end{figure}

\subsection{Example 1: Exact recovery of a single source}
Figures \ref{fig:ex1_1} and \ref{fig:ex1_2} show numerical solutions of \eqref{eq:BF2}. In these problems we recover a single source and $\bb = \AAA \ee_j$. The theory developed for Problem III is therefore applicable, see Theorem \ref{thm:Palpha}. 
Figure \ref{fig:ex1_comp}  contains a comparison of the size of $\gam$, cf. \eqref{eq:gamma}, and the maximum value, $\max_i [\xx_\alpha^*]_i$, of the solution $\xx_\alpha^*$ of Problem III. 
 We observe that the outcome of these experiments is as one could have anticipated from Theorem \ref{thm:Palpha}.

\begin{figure}[h]
    \centering
    \begin{subfigure}[b]{0.45\linewidth}        
        \centering
        \includegraphics[width=\linewidth]{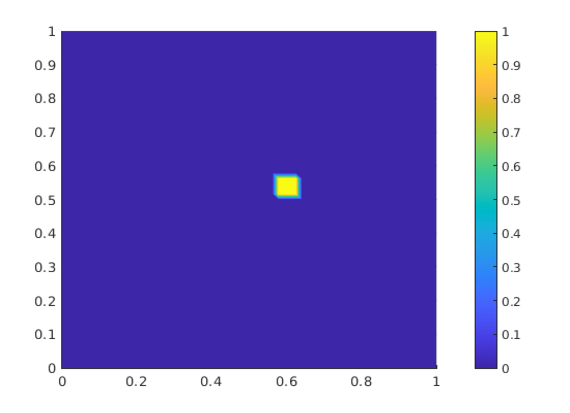}
        \caption{True source.}
    \end{subfigure}
    \begin{subfigure}[b]{0.45\linewidth}        
        \centering
        \includegraphics[width=\linewidth]{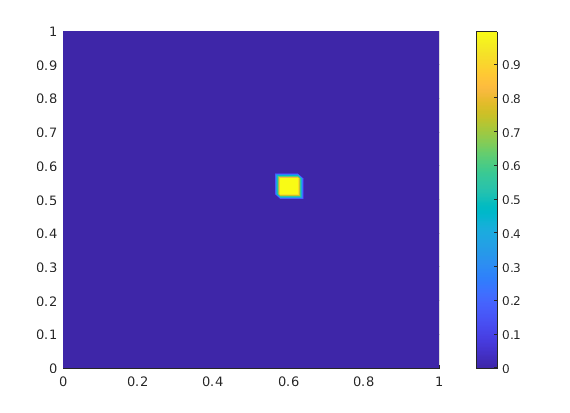}
        \caption{Inverse solution.}
    \end{subfigure}\par
    \caption{Comparison of a true \textit{interior} source and the inverse solution computed by solving \eqref{eq:BF2}, Example 1. The size of the regularization parameter was $\alpha = 10^{-4}$.}
    \label{fig:ex1_1}
\end{figure}

\begin{figure}[h]
    \centering
    \begin{subfigure}[b]{0.45\linewidth}        
        \centering
        \includegraphics[width=\linewidth]{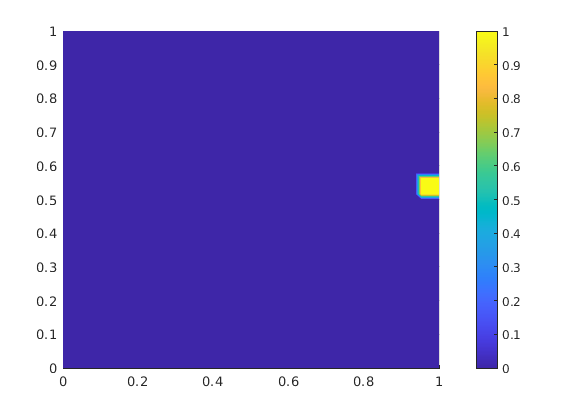}
        \caption{True source.}
    \end{subfigure}
    \begin{subfigure}[b]{0.45\linewidth}        
        \centering
        \includegraphics[width=\linewidth]{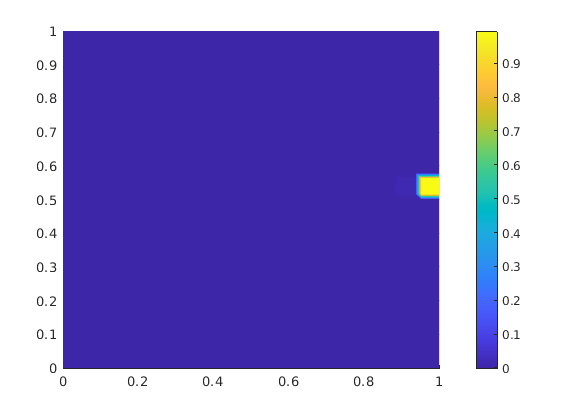}
        \caption{Inverse solution.}
    \end{subfigure}\par
    \caption{Comparison of a true source located at the \textit{boundary} and the inverse solution computed by solving \eqref{eq:BF2}, Example 1. The size of the regularization parameter was $\alpha = 10^{-3}.$}
    \label{fig:ex1_2}
\end{figure}

\begin{figure}[h]
    \centering
    \begin{subfigure}[b]{0.45\linewidth}        
        \centering
        \includegraphics[width=\linewidth]{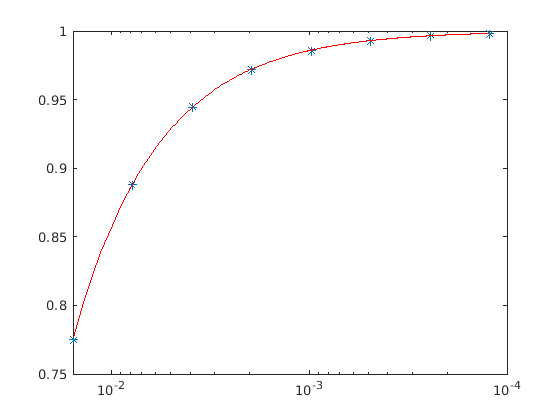}
        \caption{Interior source.}
    \end{subfigure}
    \begin{subfigure}[b]{0.45\linewidth}        
        \centering
        \includegraphics[width=\linewidth]{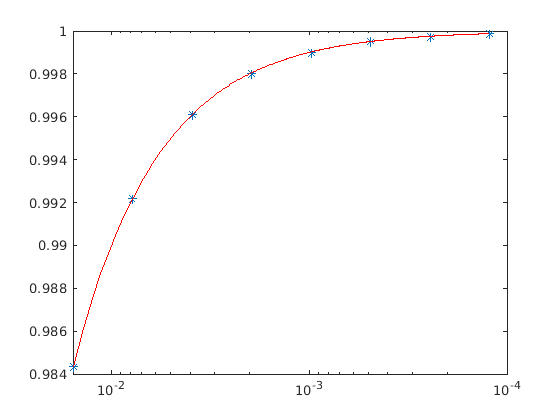}
        \caption{Source at the boundary.}
    \end{subfigure}\par
    \caption{Example 1. The red curve shows the size of $\gam$, see \eqref{eq:gamma}, as a function of the regularization parameter $\alpha$, and the asterisks represent $\max_i [\xx_\alpha^*]_i$, where $\xx_\alpha^*$ is the solution of \eqref{eq:BF2}.} 
    \label{fig:ex1_comp}
\end{figure}

\subsection{Example 2: Noise}
If the observation data contains noise, it is natural to solve \eqref{eq:BF4}. Throughout this example, $k = 7$ in the truncated SVD employed to obtain the approximation $\AAK$ of $\AAA$, see Subsection \ref{subsec:altOpt}. The experiment was executed as follows:

\begin{enumerate}
    \item Generate the data
       \begin{equation*}
           \bb = \bb^\dagger + \eta = \AAA\ee_j + \eta, \quad \eta = \delta\rho,
       \end{equation*}
       where $\delta$ is a scalar and $\rho$ is a vector containing normally distributed numbers with zero mean and standard deviation equal to $1$. See \cite[Example 6]{Elv20} for a thorough discussion of the noise.
    \item 
        Compute $\xx_k^* = \AAK^\dagger\bb = \AAK^\dagger(\AAA\ee_j + \eta)$, see \eqref{eq:BF4}.
    \item 
        Set
          \begin{equation*}
              \bar{\alpha} = \max_{i \neq j} \frac{1 + |\tau_{ij}|}{1-|\tau_{ij}|} \max_i \left|[\WWK^{-1}\AAK^\dagger\eta]_i\right|,
          \end{equation*}
        cf. Theorem \ref{thm:Pnoise}.
    \item Compute, see \eqref{eq:BF4}, 
        \begin{equation*}
            \xx_\alpha^* = \argmin_{\xx\in\Rn} \left\{ \frac{1}{2}\|\PPK\xx - \xx_k^*\|_2^2 + \alpha\|\WWK\xx\|_1 \right\}
        \end{equation*}
    for both $\alpha = 0.3\bar{\alpha}$ and $\alpha = 3\bar{\alpha}$.
\end{enumerate}
This "setup" is such that the theory, presented in Theorem \ref{thm:Pnoise}, for Problem IV is applicable. 

Note that the problem is regularized with both standard truncated SVD (i.e., the choice of the truncation parameter $k$), and $\ell^1$-regularization (i.e., the choice of $\alpha$). How to optimally choose these parameters in relation to each other is a complicated matter and left for future research.

\begin{figure}[H]
    \centering
    \includegraphics[width=0.45\linewidth]{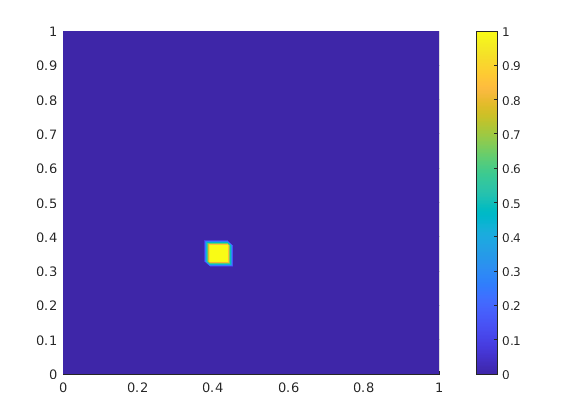}
    \caption{Example 2, true source.}\label{fig:ex2true}
\end{figure}

Figure \ref{fig:ex2a} compares over-regularized ($\alpha = 3\bar{\alpha}$) and under-regularized ($\alpha = 0.3\bar{\alpha}$) solutions of \eqref{eq:BF4} with observation data containing 5\% noise. Similar comparisons are presented in Figures \ref{fig:ex2b} and \ref{fig:ex2c} for 10\% and 15\% noise, respectively. The true source is displayed in Figure \ref{fig:ex2true}.

When $\alpha = 3\bar{\alpha}$, the plots of $\xx_\alpha^*$, displayed in panels \ref{fig:ex2a}a), \ref{fig:ex2b}a) and \ref{fig:ex2c}a), show that the true source is successfully recovered in all three cases, albeit with an underestimated magnitude. 

We can, however, as a post-processing step, improve the magnitude of the solutions using \eqref{eq:gammanoise}: Assuming $[\WWK^{-1}\AAK^\dagger\eta]_j << [\WW^{-1}\PPK\ee_j]_j$, we compute
\begin{equation*}
    \xx_{\alpha, \textnormal{SCALED}}^* = \frac{\xx_\alpha^*}{1-\frac{\alpha}{[\WWK^{-1}\PPK\ee_j]_j}},
\end{equation*}
where the index $j$ in the denominator is known since $\argmax_i \xx_\alpha^* = j$, see Theorem \ref{thm:Pnoise}. By re-scaling the solution displayed in panel (a) of Figures \ref{fig:ex2a}-\ref{fig:ex2c}, the magnitude of the rescaled source becomes 0.88, 1.0 and 0.93, respectively. This post-processing step can only be mathematically justified if the assumptions needed in Theorem \ref{thm:Pnoise} hold, but it can "always" be applied in practice. (We have not explored its success when the assumptions in Theorem \ref{thm:Pnoise} are violated.) 

When the problem is under-regularized, i.e., $\alpha = 0.3\bar{\alpha}$, the inverse solutions are still good visual approximations of the true source, see Figures \ref{fig:ex2a}b), \ref{fig:ex2b}b) and \ref{fig:ex2c}b). However, the inverse solutions also contain small contributions from other basis vectors than $\ee_j$. Both the magnitude and the number of incorrect active basis vectors appear to increase as the noise level increases.

\begin{figure}[H]
    \centering
    \begin{subfigure}[b]{0.45\linewidth}        
        \centering
        \includegraphics[width=\linewidth]{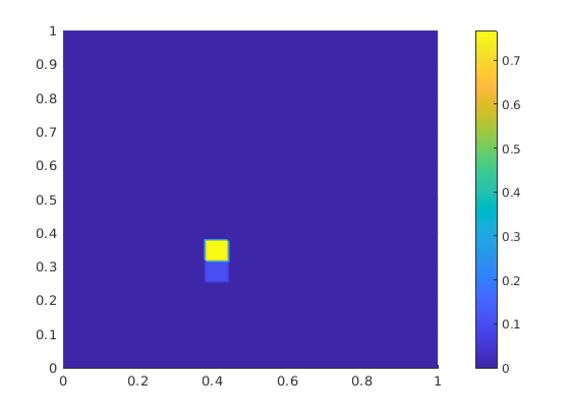}
        \caption{Inverse solution, $\alpha = 3\bar{\alpha}$.}
    \end{subfigure}
        \begin{subfigure}[b]{0.45\linewidth}        
        \centering
        \includegraphics[width=\linewidth]{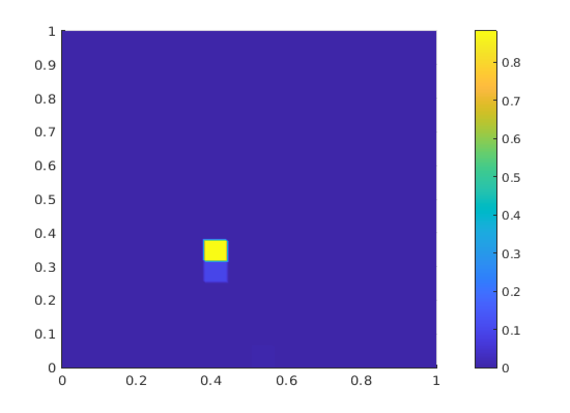}
        \caption{Inverse solution, $\alpha = 0.3\bar{\alpha}$.}
    \end{subfigure}\par
    \caption{Example 2, 5\% noise, $\bar{\alpha} = 0.0031$.}
    \label{fig:ex2a}
\end{figure}

\begin{figure}[H]
    \centering
    \begin{subfigure}[b]{0.45\linewidth}        
        \centering
        \includegraphics[width=\linewidth]{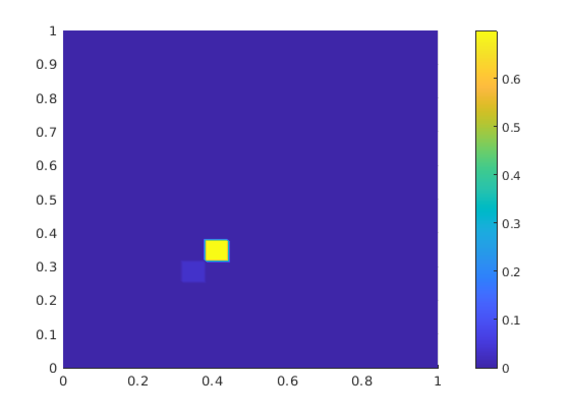}
        \caption{Inverse solution, $\alpha = 3\bar{\alpha}$.}
    \end{subfigure}
        \begin{subfigure}[b]{0.45\linewidth}        
        \centering
        \includegraphics[width=\linewidth]{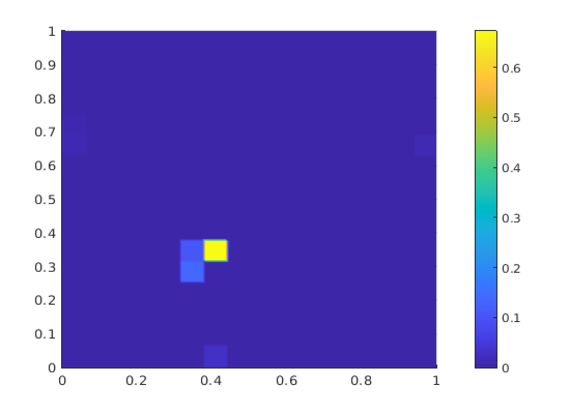}
        \caption{Inverse solution, $\alpha = 0.3\bar{\alpha}$.}
    \end{subfigure}\par
    \caption{Example 2, 10\% noise, $\bar{\alpha} = 0.0067$.}
    \label{fig:ex2b}
\end{figure}

\begin{figure}[H]
    \centering
    \begin{subfigure}[b]{0.45\linewidth}        
        \centering
        \includegraphics[width=\linewidth]{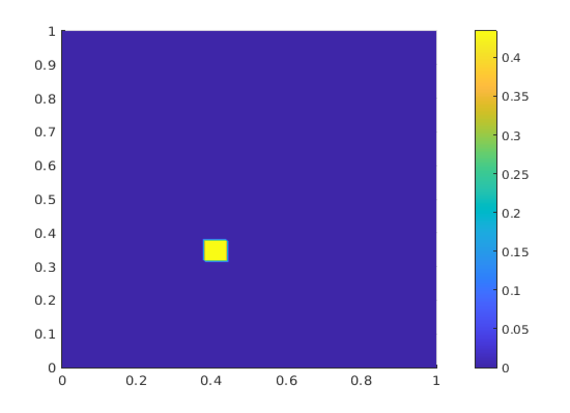}
        \caption{Inverse solution, $\alpha = 3\bar{\alpha}$.}
    \end{subfigure}
        \begin{subfigure}[b]{0.45\linewidth}        
        \centering
        \includegraphics[width=\linewidth]{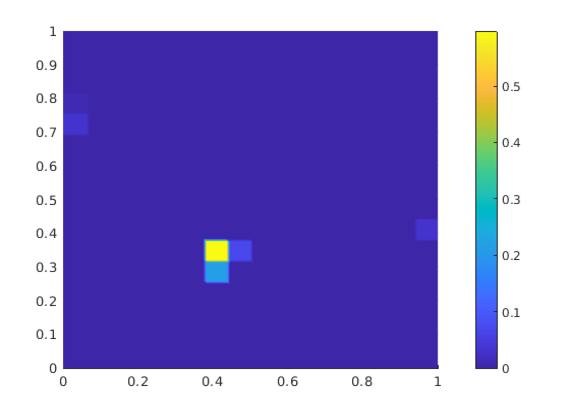}
        \caption{Inverse solution, $\alpha = 0.3\bar{\alpha}$.}
    \end{subfigure}\par
    \caption{Example 2, 15\% noise, $\bar{\alpha} = 0.0122$.}
    \label{fig:ex2c}
\end{figure}

In most applications, $\bar{\alpha}$ is not available because it requires full knowledge of the noise $\eta$. We therefore performed a numerical study under the common assumption that only (an estimate of) $\|\eta\|_2$ is known. \textcolor{black}{This allows the use of Morozov's discrepancy principle \cite{morozov1966solution,BEng96} for choosing the truncation parameter $k$ for each fixed size of $\alpha$: We did not attempt to estimate appropriate values for both $k$ and $\alpha$ simultaneously, which must thus be regarded as an open problem. Figure \ref{fig:ex2many} compares solutions of \eqref{eq:BF4} for different choices of $k$ and $\alpha$, where the choice $k = 5$ is the outcome of applying Morozov's discrepancy principle with the threshold $1.05\|\eta\|_2$.} The values $k = 3$ and $k = 15$ are chosen simply to compare the choice $k = 5$ with a smaller and larger truncation parameter.

In this particular example, we observe that using a relatively strong regularization in the truncated SVD step, i.e., choosing $k = 3$, gives good reconstruction of the source for all the tested values of $\alpha$. If the regularization by truncated SVD is reduced, i.e., when $k$ increases, it appears that $\alpha$ must be chosen more carefully to obtain a good reconstruction. 

\begin{figure}[H]
    \centering
    \begin{subfigure}[b]{0.32\linewidth}        
        \centering
        \includegraphics[width=\linewidth]{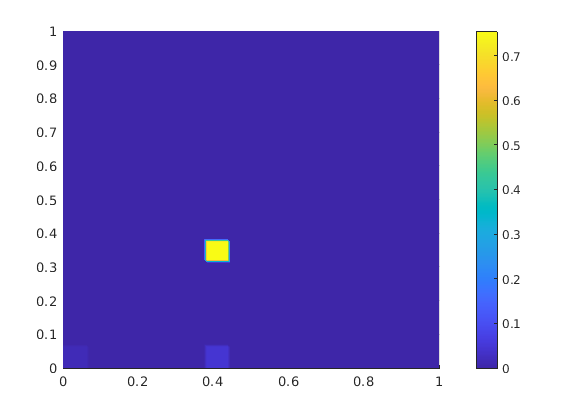}
        \caption{$k=3$, $\alpha = 10^{-2}$.}
    \end{subfigure}
    \begin{subfigure}[b]{0.32\linewidth}        
        \centering
        \includegraphics[width=\linewidth]{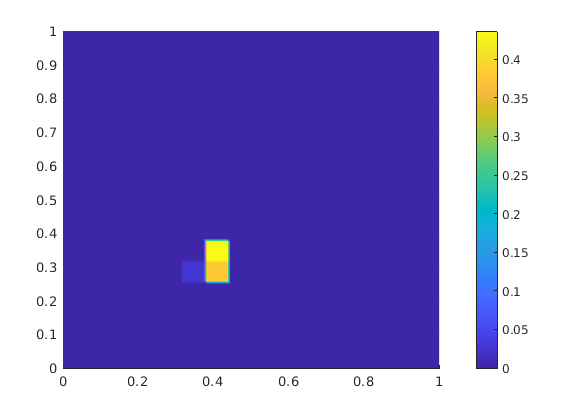}
        \caption{$k=5$, $\alpha = 10^{-2}$.}
    \end{subfigure}    
    \begin{subfigure}[b]{0.32\linewidth}        
        \centering
        \includegraphics[width=\linewidth]{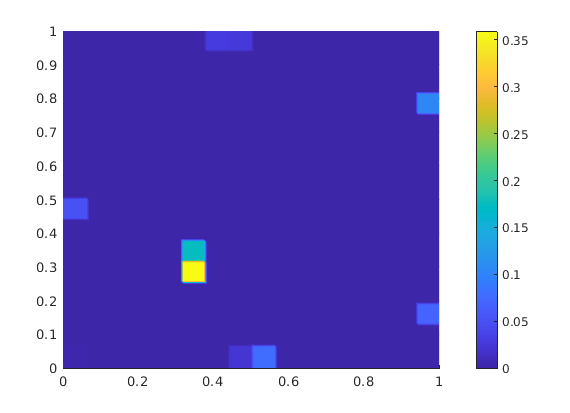}
        \caption{$k=15$, $\alpha = 10^{-2}$.}
    \end{subfigure}\par
    \begin{subfigure}[b]{0.32\linewidth}        
        \centering
        \includegraphics[width=\linewidth]{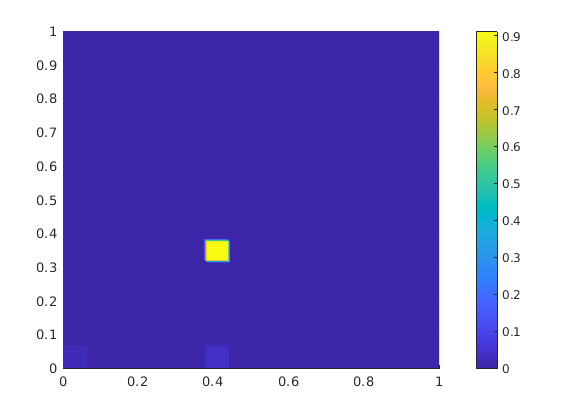}
        \caption{$k=3$, $\alpha = 10^{-3}$.}
    \end{subfigure}
    \begin{subfigure}[b]{0.32\linewidth}        
        \centering
        \includegraphics[width=\linewidth]{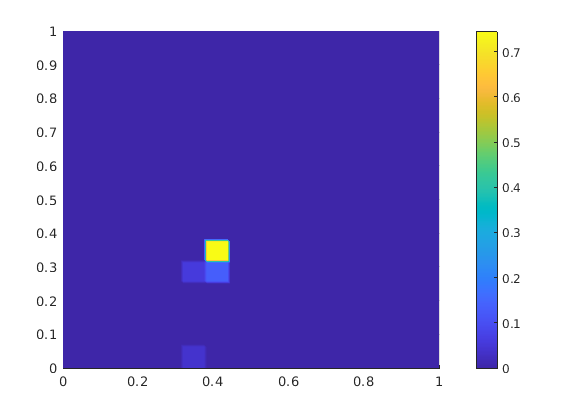}
        \caption{$k=5$, $\alpha = 10^{-3}$.}
    \end{subfigure}    
    \begin{subfigure}[b]{0.32\linewidth}        
        \centering
        \includegraphics[width=\linewidth]{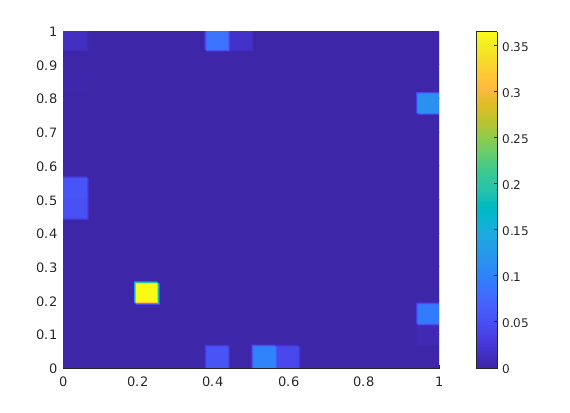}
        \caption{$k=15$, $\alpha = 10^{-3}$.}
    \end{subfigure}\par
    \begin{subfigure}[b]{0.32\linewidth}        
        \centering
        \includegraphics[width=\linewidth]{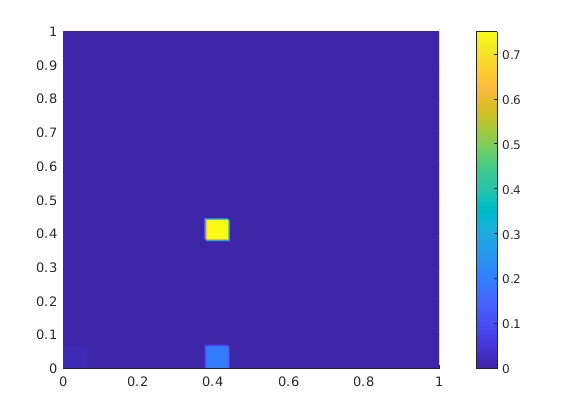}
        \caption{$k=3$, $\alpha = 10^{-4}$.}
    \end{subfigure}
    \begin{subfigure}[b]{0.32\linewidth}        
        \centering
        \includegraphics[width=\linewidth]{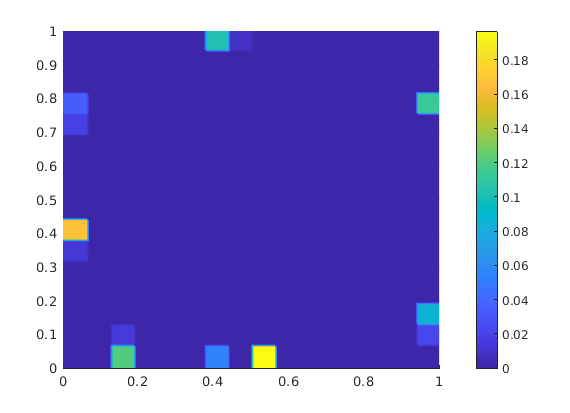}
        \caption{$k=5$, $\alpha = 10^{-4}$.}
    \end{subfigure}    
    \begin{subfigure}[b]{0.32\linewidth}        
        \centering
        \includegraphics[width=\linewidth]{images/ex2_many/k15a10000.png}
        \caption{$k=15$, $\alpha = 10^{-4}$.}
    \end{subfigure}\par    
    \caption{Example 2, 10\% noise. Comparison of inverse solutions computed with different choices of the regularization parameters $k$ and $\alpha$. \textcolor{black}{The choice $k=5$ was the outcome of using Morozov's discrepancy principle, for each given size of $\alpha$, with the threshold $1.05\|\eta\|_2$. For this test problem, the discrepancy principle lead to the same value $k=5$ for $\alpha=10^{-2}, \, 10^{-3}, \, 10^{-4}$.} Figure \ref{fig:ex2true} shows the true source.}
    \label{fig:ex2many}
\end{figure}

\subsection{Example 3: Large circular source}
\label{subsec:largeCircularSource}
So far we have considered examples covered by our analysis. We will now depart from this and explore more involved cases: The true source depicted in Figure \ref{fig:ex3true}(a) does not belong to the finite element space associated with the coarse mesh used to represent the source in the inverse computations. 

We first note that Figure \ref{fig:ex3true}(b) shows that classical Tikhonov regularization 
\begin{equation}\label{eq3tikh}
\min_{\xx\in\Rn} \left\{ \frac{1}{2}\|\AAA\xx - \bb\|_2^2 + \zeta \|\xx\|_2^2 \right\}, 
\end{equation}
with $\zeta=10^{-4}$, fails to yield an adequate solution to this problem: Compare panels (a) and (b) of Figure \ref{fig:ex3true}. The mathematical explanation for this is presented in \cite{Elv20}. 

Next, we observe in Figure \ref{fig:ex3}, left column, that the sparsity structure of the inverse solution computed by solving \eqref{eq:BF4} deteriorates as the $\ell^1$-regularization parameter $\alpha$ tends to zero. However, for all three values of $\alpha$, the position of the source is quite well recovered. In these simulations, the truncated SVD employed to obtain the approximation $\AAA_k^\dagger$ of the pseudo inverse $\AAA^\dagger$ in \eqref{eq:BF4}, was obtained by choosing the truncation parameter $k=5$. 

\textcolor{black}{As an alternative to the truncated SVD, we also used standard Tikhonov regularization to obtain an approximation of $\AAA^\dagger$:} Employing the singular value decomposition $\AAA = \mathsf{U}\Sigma\mathsf{V}^T$ of $\AAA$, the solution of the minimization problem
\begin{equation*}
    \min_{\hxx\in\Rn} \left\{ \frac{1}{2}\|\AAA\hxx-\bb\|_2^2 + \frac{1}{2}\halpha\|\hxx\|_2^2 \right\},
\end{equation*}
can be expressed as
\begin{equation*}
  \hxx_\halpha = \mathsf{V}(\Sigma^2 + \halpha\mathsf{I})^{-1}\mathsf{V}^T \AAA^T \bb 
  = \SSa\bb,
\end{equation*}
where 
\begin{equation}\label{eq:Sa}
\SSa = \mathsf{V}(\Sigma^2 + \halpha\mathsf{I})^{-1}\mathsf{V}^T \AAA^T \approx \AAA^\dagger. 
\end{equation}
Replacing $\AAA^\dagger$ in \eqref{eq:BF2} with $\SSa$, keeping in mind that $\PP = \AAA^\dagger\AAA$, leads to the following alternative to \eqref{eq:BF4}
      \begin{equation} \label{eq:OLE5}
        \min_{\xx\in\Rn} \left\{ \frac{1}{2}\|\SSa\AAA\xx - \SSa \bb\|_2^2 + \alpha\|\WW\xx\|_1 \right\}.
      \end{equation}
\textcolor{black}{We also mention that, so far, we have not been able to modify the proof of Theorem \ref{thm:Pnoise} to Tihonov based approximations of $\AAA^\dagger$, i.e., to \eqref{eq:OLE5} with $\bb=\AAA \ee_j + \eta$.}

\begin{figure}[H]
    \centering
    \begin{subfigure}[b]{0.45\linewidth}        
        \centering
        \includegraphics[width=\linewidth]{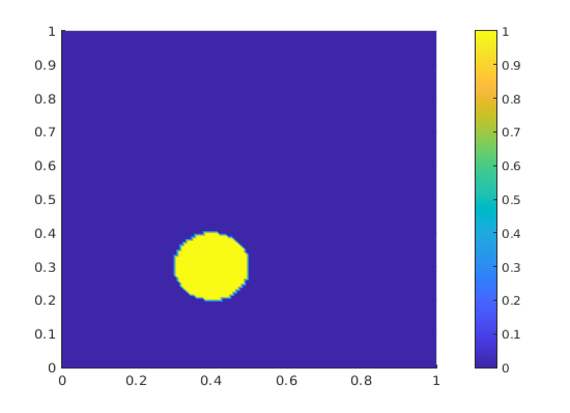}
        \caption{True source.}
    \end{subfigure}
        \begin{subfigure}[b]{0.45\linewidth}        
        \centering
        \includegraphics[width=\linewidth]{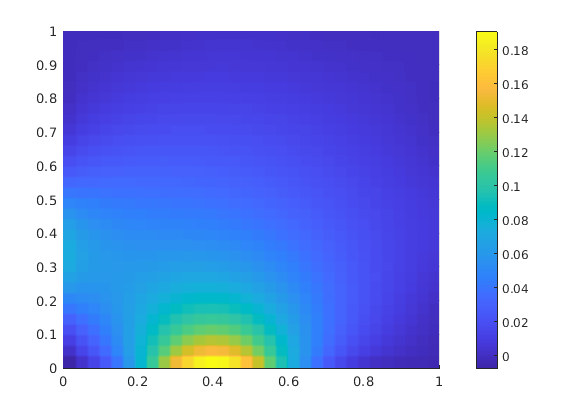}
        \caption{Solution of \eqref{eq3tikh} with $\zeta = 10^{-4}$.}
    \end{subfigure}\par
    \caption{Example 3. Panel (a) shows the true source, and panel (b) displays the numerical solution of \eqref{eq1}-\eqref{eq2} using Tikhonov regularization $\zeta \| f \|_{L^2(\Omega)}^2$ instead of the weighted $\ell^1$-regularization $\alpha \sum_i w_i |(f, \phi_i)_{L^2(\Omega)}|$.}
    \label{fig:ex3true}
\end{figure}

The right column in Figure \ref{fig:ex3} contains the results obtained by solving \eqref{eq:OLE5}. 
We observe, particularly for the two largest values of the $\ell^1$-regularization parameter $\alpha$, that truncated SVD and Tikhonov regularization yield visually rather similar results. 


\subsection{Example 4: Multiple sources}
Our last examples concern several sources. To solve the problems, we did the following:
\begin{enumerate}
    \item For $N$ true sources, we computed
       \begin{equation*}
           \bb^\dagger = \AAA\left(\sum_{i=1}^N \ee_{q_i}\right).
       \end{equation*}
    \item Thereafter, we solved the problem 
        \begin{equation*}
            \min_{\xx\in\Rn} \left\{ \frac{1}{2}\|\SSa\AAA\xx - \SSa\bb^\dagger\|_2^2 + \alpha\|\WW\xx\|_1 \right\},
        \end{equation*}
        where $\SSa$ is defined in \eqref{eq:Sa}.
\end{enumerate}

Figure \ref{fig:ex4} contains results obtained with 2, 4 and 8 true sources. Panel b) shows that the two sources are nearly perfectly recovered. For the case with 4 sources, the inverse solution recovers three of the sources almost perfectly, while the fourth is 'split' into two adjacent sources with less magnitude, cf. Panel d). Panel f) shows the inverse solution for the case with 8 true sources, where we observe that the three sources located at the boundary are recovered very well (note the color bar), whereas the interior sources are merged somewhat into two clusters in the inverse solution.

\section{Conclusions}
If the exact data is generated from a single basis vector $\ee_j$, our weighted $\ell^1$-regularization technique is able to exactly recover the true solution. When noise is present, we have obtained estimates for the size of the regularization parameter $\alpha$ which yield an inverse solution in the form $\gamn\ee_j$, where $\gamn$ is a positive scalar. Numerical experiments suggest that our method also can, in many cases, identify several local sources, but we do not have a thorough mathematical understanding of this. We only know with certainty that it is possible to construct scenarios for which our scheme will fail to recover two sources. 

\textcolor{black}{The computation of our weight matrix involves the pseudo inverse, which in practise must be approximated by a more "well-behaved" operator. In the analysis we accomplished this by employing a truncated SVD approach, and we observed numerically that also Tikhonov based approximations of the Moore-Penrose inverse work well. For the latter, however, it remains to develop a rigorous mathematical analysis.} 

Concerning the practical use of our weighted $\ell^1$-regularization method, it seems reasonable to expect that the method can recover two or three well-separated sources. Nevertheless, the definition of "well-separated" is problem dependent since it must depend on the smoothing properties of the involved forward operator and the geometry of the solution domain $\Omega$. This means that one should, for each concrete application, perform a simulation study to explore which source patterns that can be identified by the weighted $\ell^1$-regularization method.  

We defined our regularization operator and presented our analysis in terms of Euclidean spaces. Consequently, the methodology can be applied whenever a discrete version of a source identification task can be formulated in terms of a transfer matrix with a significant null space. For example, the use is not restricted to PDE-constrained optimization problems with elliptic state equations, but can also be applied when the state equation is parabolic or hyperbolic. 

This study was motivated by inverse problems arising in connection with EEG and ECG recordings. In principle, our scheme can be applied to these problems, but a number of challenging engineering issues must be handled: one must construct suitable geometrical models (using, e.g., MR images), obtain EEC or ECG recordings, handle noisy data, \textcolor{black}{construct suitable basis functions for the source term which enables the incorporation of dipoles,} etc. We intend to explore the EEG and ECG applications in forthcoming investigations. 

\textcolor{black}{In this paper we search for a source in a finite dimensional space. From a pure mathematical perspective, an interesting problem would be to develop an analogous theory using an infinite dimensional source space. The authors believe that this {\em might} be possible following the approach presented in \cite{casas2012approximation}.}

\begin{figure}[H]
    \centering
    \begin{subfigure}[b]{0.45\linewidth}        
        \centering
        \includegraphics[width=\linewidth]{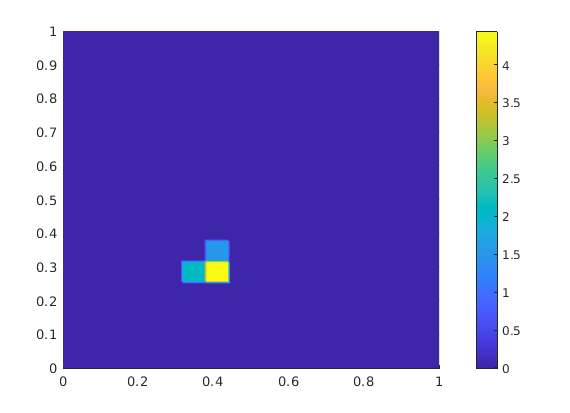}
        \caption{Truncated SVD. $\alpha = 10^{-2}$.}
    \end{subfigure}
    \begin{subfigure}[b]{0.45\linewidth}        
        \centering
        \includegraphics[width=\linewidth]{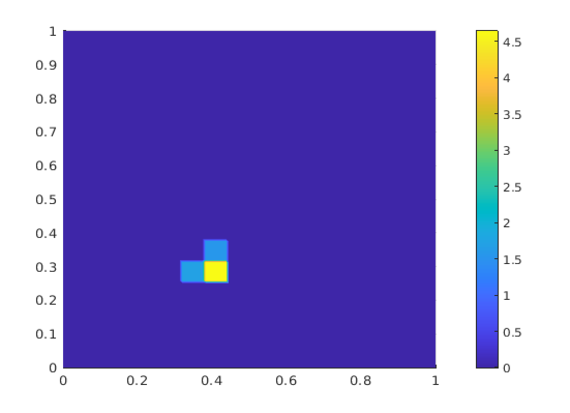}
        \caption{Tikhonov. $\alpha = 10^{-2}$.}
    \end{subfigure}\par
    \begin{subfigure}[b]{0.45\linewidth}        
        \centering
        \includegraphics[width=\linewidth]{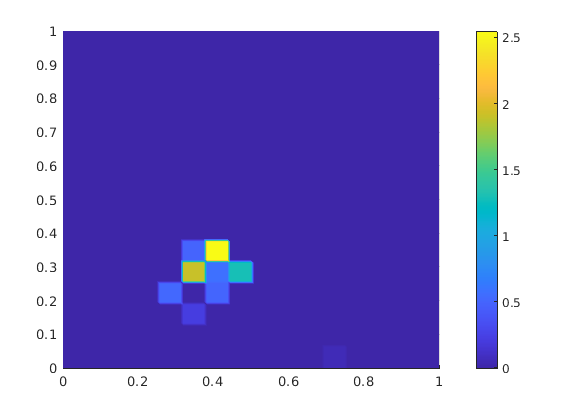}
        \caption{Truncated SVD. $\alpha = 10^{-4}$.}
    \end{subfigure}
    \begin{subfigure}[b]{0.45\linewidth}        
        \centering
        \includegraphics[width=\linewidth]{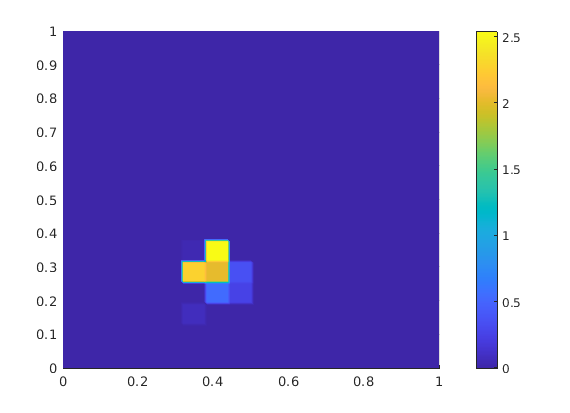}
        \caption{Tikhonov. $\alpha = 10^{-4}$.}
    \end{subfigure}\par
    \begin{subfigure}[b]{0.45\linewidth}        
        \centering
        \includegraphics[width=\linewidth]{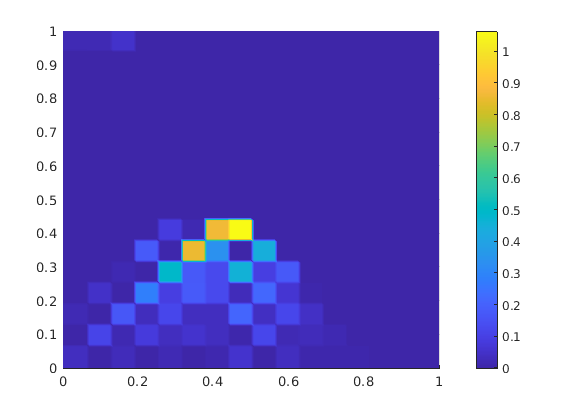}
        \caption{Truncated SVD. $\alpha = 10^{-6}$.}
    \end{subfigure}
    \begin{subfigure}[b]{0.45\linewidth}        
        \centering
        \includegraphics[width=\linewidth]{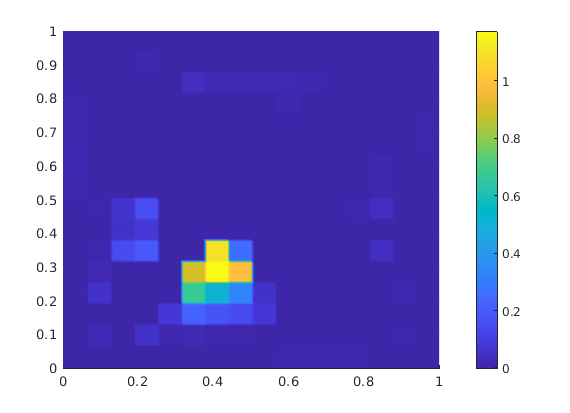}
        \caption{Tikhonov. $\alpha = 10^{-6}$.}
    \end{subfigure}\par    
    \caption{Example 3, employing weighted $\ell^1$-regularization $\alpha \| \WW \xx \|_1$. The inverse solutions were computed by solving \eqref{eq:BF4} (approximating $\AAA^\dagger$ with truncated SVD, $k = 5$) and \eqref{eq:OLE5} (approximating $\AAA^\dagger$ with Tikhonov regularization, $\beta = 10^{-6}$). The true source is displayed in Figure \ref{fig:ex3true}(a).}
    \label{fig:ex3}
\end{figure}

\begin{figure}[H]
    \centering
    \begin{subfigure}[b]{0.45\linewidth}        
        \centering
        \includegraphics[width=\linewidth]{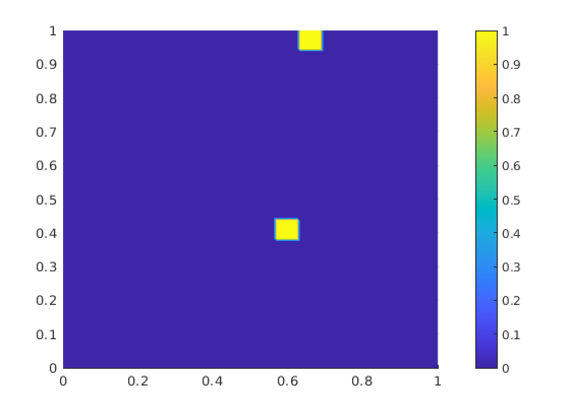}
        \caption{True sources.}
    \end{subfigure}
    \begin{subfigure}[b]{0.45\linewidth}        
        \centering
        \includegraphics[width=\linewidth]{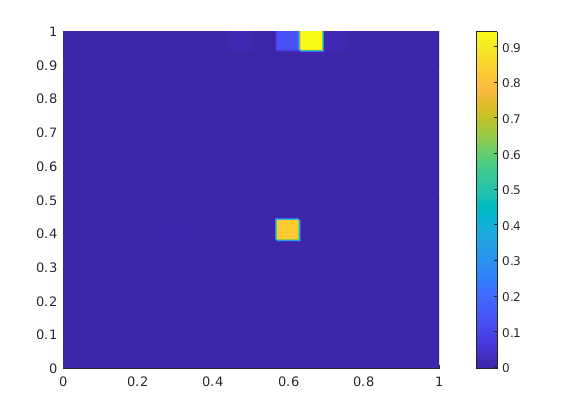}
        \caption{Inverse solution.}
    \end{subfigure}\par
        \begin{subfigure}[b]{0.45\linewidth}        
        \centering
        \includegraphics[width=\linewidth]{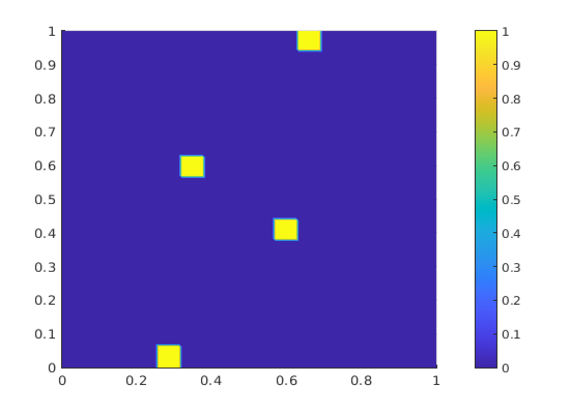}
        \caption{True sources.}
    \end{subfigure}
    \begin{subfigure}[b]{0.45\linewidth}        
        \centering
        \includegraphics[width=\linewidth]{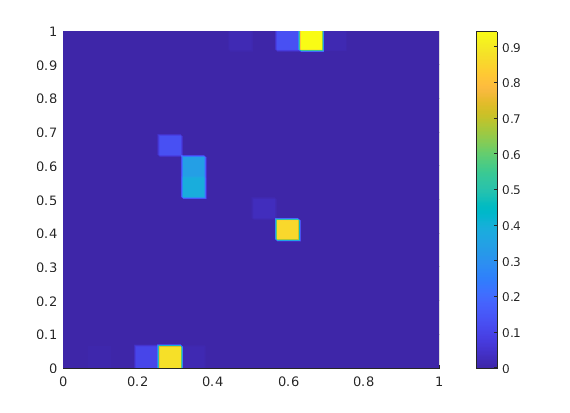}
        \caption{Inverse solution.}
    \end{subfigure}\par
        \begin{subfigure}[b]{0.45\linewidth}        
        \centering
        \includegraphics[width=\linewidth]{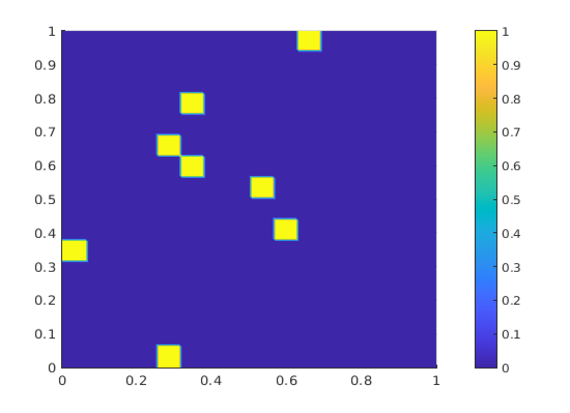}
        \caption{True sources.}
    \end{subfigure}
    \begin{subfigure}[b]{0.45\linewidth}        
        \centering
        \includegraphics[width=\linewidth]{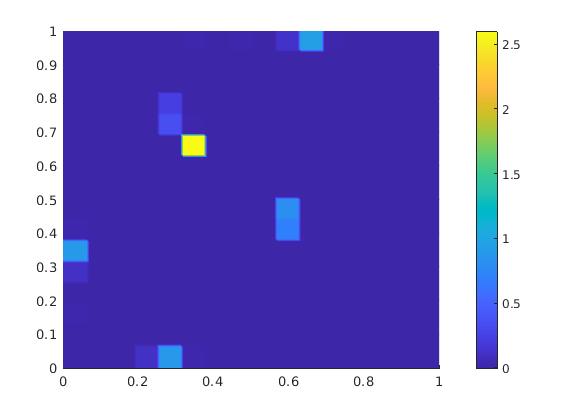}
        \caption{Inverse solution.}
    \end{subfigure}\par
    \caption{Example 4, recovering several local sources. Here, $\alpha = 0.01$ and $\halpha = 10^{-6}$ in all the three experiments.}
    \label{fig:ex4}
\end{figure}

\appendix

\section{\textcolor{black}{Standard sparsity regularization}} \label{app:standard}
\textcolor{black}{
We observed in panel b) of Figure \ref{fig:ex0} that standard sparsity regularization failed to recover an interior source. We will now explore this issue in some detail.}

\textcolor{black}{
Recall the definition \eqref{eq:pre2} of the forward operator $K_h: F_h \rightarrow L^2(\partial\Omega)$ and consider the problem
\begin{equation}
    \min_{f \in F_h}{\sum_i |(f,\phi_i)|} \quad \textnormal{subject to} \quad K_h f = K_h\phi_j.\label{eq:app1}
\end{equation}
This is the basis pursuit problem associated with \eqref{eq1}-\eqref{eq2} when $d=K_h\phi_j$, provided that $w_i = 1$ for $i=1,2,\ldots,n$, i.e., with standard unweighted sparsity regularization. We assume in this appendix that the basis functions $\phi_1, \phi_2, \ldots, \phi_n$ for $F_h$ satisfy  
\begin{equation} \label{eq:app0}
    \| \phi_1 \|_\infty = \| \phi_2 \|_\infty = \ldots = \| \phi_n \|_\infty.  
\end{equation}
}

\textcolor{black}{
Let us define the orthogonal projection
\begin{equation*}
    P_h: F_h \rightarrow \mathcal{N}(K_h)^\perp, 
\end{equation*}
where $\mathcal{N}(K_h)$ denotes the null space of $K_h$ and we employ the standard $L^2$-inner product on $F_h \subset L^2(\Omega)$.
Since $P_h$ and $K_h$ have the same null space, we can reformulate the basis pursuit problem \eqref{eq:app1} as 
\begin{equation}
    \min_{f \in F_h}{\sum_i |(f,\phi_i)|} \quad \textnormal{subject to} \quad P_h f = P_h\phi_j.\label{eq:app2}
\end{equation}
The associated Lagrangian $\mathcal{L}: F_h \times F_h \rightarrow \mathbb{R}$ reads 
\begin{equation*}
    \mathcal{L}(f, \lambda) = {\sum_i |(f,\phi_i)|} + (\lambda, P_h\phi_j - P_h f), 
\end{equation*}
and the Lagrange conditions become
\begin{align}
  P_h\lambda &\in \partial_f \left(\sum_i |(f,\phi_i)|\right), \label{eq:app2.1}\\
  P_h f &= P_h\phi_j. \nonumber
\end{align}}

\textcolor{black}{Since $f \rightarrow |(f,\phi_i)|$ is convex for $i=1,2,\ldots,n$, it follows that 
\begin{equation*}
    \partial_f \left(\sum_i |(f,\phi_i)|\right) = \sum_i \partial_f |(f,\phi_i)|,
\end{equation*}
provided that one interprets the right-hand-side in terms of the Minkowski sum of sets.
According to the definition of the subgradient, $q \in \partial_f |(f,\phi_i)|$ if
\begin{equation} \label{eq:app2.2}
    |(g,\phi_i)| \geq |(f,\phi_i)| + (q, g - f), \quad \forall g \in F_h.
\end{equation}
By expanding $f, g$ and $q$ in the orthonormal $F_h$-basis, i.e., 
\begin{align*}
    f(x) &= \sum_k f_k \, \phi_k (x), \\
    g(x) &= \sum_k g_k \, \phi_k (x), \\
    q(x) &= \sum_k q_k \, \phi_k (x), 
\end{align*}
we get from \eqref{eq:app2.2} that 
\begin{equation*}
    |g_i| \geq |f_i| + \sum_k q_k (\phi_k, g - f), \quad \forall g \in F_h.
\end{equation*}
Note that this implies that $q_k = 0$ for $k \neq i$. Consequently, $q (x) = q_i \phi_i (x)$, and $q_i$ must obey the inequality constraint
\begin{equation*}
    |g_i| \geq |f_i| + q_i(g_i - f_i), \quad \forall g \in F_h,
\end{equation*}
which implies that 
\begin{equation*} 
    q_i \in \begin{cases}
        \{1\}, & f_i > 0, \\
        \{-1\}, & f_i < 0, \\
        [-1,1], & f_i = 0.
    \end{cases}
\end{equation*}
Thus, we can write \eqref{eq:app2.1} as
\begin{equation} \label{eq:app3}
    [P_h\lambda]_i \in \begin{cases}
        \{1\}, & f_i > 0, \\
        \{-1\}, & f_i < 0, \\
        [-1,1], & f_i = 0,
    \end{cases}
\end{equation}
where we use the notation $[P_h\lambda]_i = (P_h\lambda,\phi_i)$.
}

\textcolor{black}{
Assume that 
\begin{equation} \label{eq:app4}
    f^*(x) = \sum_i f_i^* \phi_i(x)
\end{equation}
is a solution of \eqref{eq:app2} with associate Lagrange multiplier $\lambda^*$. Then $f^*$ and $\lambda^*$ satisfy \eqref{eq:app2.1} and from \eqref{eq:app3} we find that:} 
\textcolor{black}{
\begin{description}
\item{(a)} $-1 \leq [P_h\lambda^*]_i \leq 1$, for $i=1,2,\ldots,n$. 
\item{(b)} If a basis function $\phi_i$ with support strictly in the interior of $\Omega$ is present in the solution \eqref{eq:app4} and $f_i^* > 0$, then $[P_h\lambda^*]_i = 1$. That is, $P_h\lambda^*$ attains its maximum in the interior region associated with $\phi_i$, provided that \eqref{eq:app0} holds.
\item{(c)}
On the other hand, from the analysis presented in Section 2 in \cite{Elv20}, we know that the infinite-dimensional counterpart\footnote{\textcolor{black}{That is, the orthogonal projection when $F_h$ is replaced with $L^2(\Omega)$.}} $P\lambda^*$ to $P_h\lambda^*$ satisfies 
\begin{equation*}
    - \Delta P\lambda^* + \epsilon P\lambda^* = 0. 
\end{equation*}
It thus follows from classical maximum principles that $P\lambda^*$ can {\em not} attain a (non-negative) maximum in the interior of $\Omega$. This is not compatible/consistent with $P_h\lambda^*$ attaining its maximum in the interior which, according to (b), would be case if the solution $f^*$ is positive in an interior region. 
Hence, we expect that $f^*$ only can be positive close to the boundary $\partial \Omega$ of $\Omega$, cf. panel b) in Figure \ref{fig:ex0}.
\end{description}
}

\section{\textcolor{black}{Discretization of the state equation}} 
\label{app:discretization}
\textcolor{black}{We will briefly explain how \eqref{eq:pre1} can be discretized, using the finite element method, and thereby obtain an expression for the matrix $\tilde{\mathsf{K}}$ in \eqref{eq:pre4}. As mentioned in the numerical experiments section, the state $u(x)=\sum_k u_k N_k(x)$ and the source $f(x)=\sum_i f_i \, \phi_i(x)$ were discretized in terms of first order Lagrange elements and the characteristic functions \eqref{eq:revision2}, respectively.} 

\textcolor{black}{The discrete matrix-vector version of \eqref{eq:pre1} reads 
\begin{equation*}
    \mathsf{L}\mathbf{u} + \epsilon\mathsf{M}\mathbf{u} = \tilde{\mathsf{M}}\mathbf{f},
\end{equation*}
where $\mathsf{L}$ and $\mathsf{M}$ denote the standard stiffness and mass matrices, respectively, and 
\begin{equation*}
    \tilde{\mathsf{M}} = [\tilde{m}_{ki}], \quad \tilde{m}_{ki} = (\phi_i,N_k)_{L^2(\Omega)}. 
\end{equation*}
Hence, 
\begin{equation*}
    \mathbf{u} = [\mathsf{L} + \epsilon\mathsf{M}]^{-1} \tilde{\mathsf{M}} \mathbf{f},
\end{equation*}
and, if we discretize the fidelity term in \eqref{eq1} and combine it with this expression for $\mathbf{u}$, we get
\begin{equation*}
    \frac{1}{2} (\mathbf{u}-\mathbf{d})^T\mathsf{M}_\partial(\mathbf{u}-\mathbf{d}) =
    \frac{1}{2} \left([\mathsf{L} + \epsilon\mathsf{M}]^{-1} \tilde{\mathsf{M}}\mathbf{f} - \mathbf{d}\right)^T\mathsf{M}_\partial \left([\mathsf{L} + \epsilon\mathsf{M}]^{-1} \tilde{\mathsf{M}}\mathbf{f}-\mathbf{d} \right). 
\end{equation*}
Since the "boundary mass matrix" $\mathsf{M}_\partial$ is symmetric and positive semi-definite, we can take the square root of it to obtain the following Euclidean form of the fidelity term 
\begin{equation*}
    \frac{1}{2} \left\|\mathsf{M}_\partial^{\frac{1}{2}}[\mathsf{L} + \epsilon\mathsf{M}]^{-1} \tilde{\mathsf{M}} \mathbf{f} - \mathsf{M}_\partial^{\frac{1}{2}}\mathbf{d}\right\|_2^2 
    =  \frac{1}{2} \left\| \mathsf{M}_\partial^{\frac{1}{2}} \tilde{\mathsf{K}} \mathbf{f} - \mathsf{M}_\partial^{\frac{1}{2}}\mathbf{d} \right\|_2^2,
\end{equation*}
where 
\begin{align*}
\tilde{\mathsf{K}} &:= [\mathsf{L} + \epsilon\mathsf{M}]^{-1} \tilde{\mathsf{M}} .
\end{align*}}

\bibliographystyle{abbrv}
\bibliography{references}

\end{document}